\documentclass[reqno, 11pt]{amsart}

\makeatletter
\let\origsection=\section \def\section{\@ifstar{\origsection*}{\mysection}} 
\def\mysection{\@startsection{section}{1}\z@{.7\linespacing\@plus\linespacing}{.5\linespacing}{\normalfont\scshape\centering\S}}
\makeatother  

\usepackage{amsmath}
\usepackage{amssymb}
\usepackage{amsthm}
\usepackage{mathtools}
\usepackage{letterswitharrows}
\usepackage{tikz-cd}
\usepackage[only,llbracket,rrbracket]{stmaryrd}
\usepackage[linesnumbered,ruled,vlined]{algorithm2e}
\usepackage{enumitem}
\setenumerate{label={\normalfont (\roman*)}}

\usepackage[utf8]{inputenc}
\usepackage[T1]{fontenc}
\usepackage{lmodern}
\usepackage[babel]{microtype}
\usepackage[english]{babel}
\usepackage{relsize}

\usepackage{graphicx}
\usepackage{subcaption}
\usepackage{import}

\usepackage{svg}
\usepackage[extractname=filename]{svg-extract}

\usepackage{geometry}
\usepackage{xcolor} 
\colorlet{darkishRed}{red!80!black}
\colorlet{darkishBlue}{blue!60!black}
\colorlet{darkishGreen}{green!60!black}
\usepackage{hyperref}
\hypersetup{
	colorlinks,
	linkcolor={blue!60!black},
	citecolor={green!60!black},
	urlcolor={blue!60!black},
}
\usepackage[abbrev, msc-links]{amsrefs}
\usepackage[nameinlink, capitalise, noabbrev]{cleveref}
\crefformat{enumi}{#2#1#3}
\crefformat{equation}{#2(#1)#3}
\crefname{mainresult}{Theorem}{Theorems}
\usepackage{doi}

\renewcommand{\PrintDOI}[1]{\doi{#1}}

\let\setminus=\smallsetminus

\renewcommand{\subset}{\subseteq}

\renewcommand{\leq}{\leqslant}
\renewcommand{\geq}{\geqslant}

\renewcommand{\le}{\leq}

\newtheorem{theorem}{Theorem}[section]

\newtheorem{lemma}[theorem]{Lemma}

\newtheorem{observation}[theorem]{Observation}

\newtheorem{mainresult}{Theorem} 
\newtheorem{maincorollary}[mainresult]{Corollary}

\newcounter{claimcounter}[theorem]
\newtheorem{claim}[claimcounter]{Claim}
\newtheorem*{claim*}{Claim}
\crefname{claim}{Claim}{Claims}

\newcounter{subclaimcounter}[claimcounter]
\newtheorem*{subclaim*}{Subclaim}

\theoremstyle{definition}

\crefname{mainexample}{Example}{Examples}
\newtheorem{example}[theorem]{Example}
\crefname{example}{Example}{Examples}

\newtheorem{algo}[theorem]{Algorithm}

\crefname{routine}{Routine}{Routines}

\crefname{subroutine}{Subroutine}{Subroutines}

\crefname{subsubroutine}{Subsubroutine}{Subsubroutines}

\crefname{step}{Step}{Steps}
\theoremstyle{remark}

\newcommand{\COMMENT}[1]{{}}

\let\eps=\varepsilon
\let\epsilon=\varepsilon
\let\theta=\vartheta
\let\rho=\varrho
\let\phi=\varphi
\def\N{\mathbb N}

\makeatletter

\def\calCommandfactory#1{%
  \expandafter\def\csname c#1\endcsname{\mathcal{#1}}}
\def\frakCommandfactory#1{%
  \expandafter\def\csname frak#1\endcsname{\mathfrak{#1}}}
   
\newcounter{ctr}
\loop
  \stepcounter{ctr}
  \edef\X{\@Alph\c@ctr}
  \expandafter\calCommandfactory\X
  \expandafter\frakCommandfactory\X
\ifnum\thectr<26
\repeat

\usepackage{etoolbox}

\newbool{arXiv}
\booltrue{arXiv} 

\newcommand{\arXivOrNot}[2]{\ifbool{arXiv}{{#1}}{{#2}}}

\newbool{pdfBool}
\booltrue{pdfBool} 

\newcommand{\pdfOrNot}[2]{\ifbool{pdfBool}{{#1}}{{#2}}}

\makeatletter
\newcommand\thankssymb[1]{\textsuperscript{\@fnsymbol{#1}}}
\makeatother
\newcommand{\lk}{({<}\,k)}

\newcommand{\lek}{({\le}\,k)}

\newcommand{\lA}{({<}\,\aleph_0)}
\newcommand{\td}{tree-decom\-posi\-tion}

\newcommand{\up}{{\uparrow}}

\newcommand{\strictup}{ \mathring{\uparrow} }

\DeclareMathOperator{\Dom}{Dom}

\DeclareMathOperator{\rt}{root}

\DeclareMathOperator{\interior}{int}
\newcommand{\size}{|\!\cdot\!|}

\newcounter{mylabelcounter}

\makeatletter
\newcommand{\labelText}[2]{%
#1\refstepcounter{mylabelcounter}%
\immediate\write\@auxout{%
  \string\newlabel{#2}{{1}{\thepage}{{\unexpanded{#1}}}{mylabelcounter.\number\value{mylabelcounter}}{}}%
}%
}

\makeatletter
\newcommand\footnoteref[1]{\protected@xdef\@thefnmark{\ref{#1}}\@footnotemark}
\makeatother

\definecolor{cMaroon}{HTML}{93152a}
\newcommand{\defn}[1]{{\color{cMaroon}{\emph{#1}}}}
\newcommand{\defnm}[1]{{\color{cMaroon}{#1}}}

\DeclareMathOperator{\tht}{ht}

\DeclareMathOperator{\is}{\interior(\sigma)}

\DeclareMathOperator{\cl}{cl}

\usepackage{tabularx}

\title{Displaying prescribed sets of ends by linked tree-decompositions}

\author[S.~Albrechtsen]{Sandra Albrechtsen\thankssymb{1}}
\thanks{\thankssymb{1} Supported by the Alexander von Humboldt Foundation in the framework of the Alexander von Humboldt Professorship of Daniel Král' endowed by the Federal Ministry of Education and Research.}
\address{(S.~Albrechtsen) Leipzig University, Institute of Mathematics, Augustusplatz 10, 04109 Leipzig, Germany}

\author[M.~Pitz]{Max Pitz}

\author[R.~Schaut]{Roman Schaut\thankssymb{2}}
\thanks{\thankssymb{2} Supported by a State Graduate Funding Program scholarship of the University of Hamburg.}

\address{(M.~Pitz and R.~Schaut) University of Hamburg, Department of Mathematics, Bundesstraße 55 (Geomatikum), 20146 Hamburg, Germany}


\begin{document}

\begin{abstract}
We show that if a subset  $\Psi$ of the ends of a graph $G$ can be displayed by a tree-decomposition of finite adhesion, then it can also be displayed by a linked such tree-decomposition. This tree-decomposition captures all combinatorial information of the ends in $\Psi$: their degrees, their sets of dominating vertices, and their combined degrees.
\end{abstract}

\vspace*{-24pt}

\maketitle

\section{Introduction} \label{sec:Intro}

Following Thomas \cite{thomas1989wqo}, we say a rooted \td\ $(T, \{V_t \colon t\in T\})$ of a (possibly infinite) graph $G$ is \defn{linked} if for any two comparable nodes $s < t \in T$, the maximum number of disjoint $V_s{-}V_t$ paths in $G$ equals the size of a smallest adhesion set corresponding to an edge on the $s{-}t$ path in $T$. (For formal definitions, see \cref{section:Background}). 
For finite graphs, linked tree-decompositions have been instrumental in Robertson and Seymour's proof of the graph minor theorem \cite{ROBERTSON1990227}. 
For infinite graphs, the first seminal result regarding linked tree-decompositions has been the following theorem due to  K{\v r}{\'i}{\v z} and Thomas, which formed the basis for Thomas's proof \cite{thomas1989wqo} that the (infinite) graphs of tree-width $<k \in \N$ are better-quasi-ordered under the minor relation.

\begin{theorem}[{K{\v r}{\'i}{\v z} and Thomas 1991  \cite{kriz1991mengerlikepropertytreewidth}}]
\label{thm_intro_krizthomas}
 For any graph $G$ and $k \in \N$, the following statements are equivalent:
\begin{enumerate}
    \item $G$ admits a linked, rooted tree-decomposition into parts of size $<k $.
      \item $G$ admits a tree-decomposition into parts of size $<k $.
\end{enumerate}
\end{theorem}

 The following theorem generalizes \cref{thm_intro_krizthomas} from graphs with tree-decompositions into parts of uniformly bounded finite size to graphs with tree-decompositions into  (arbitrarily large) finite parts:   

\begin{theorem}[{Albrechtsen, Jacobs, Knappe and Pitz 2024~\cite{LinkedTDInfGraphs}}]\label{main:LinkedTightCompTreeDecompnew}

   For any graph $G$, the following statements are equivalent:
\begin{enumerate}
     \item \label{itm:LinkedTD:Linked} $G$ admits a linked, tight, and componental rooted tree-decomposition into finite parts.
      \item $G$ admits a tree-decomposition into finite parts.
    \item Every component of $G$ has a normal spanning tree.
\end{enumerate}
\end{theorem}

Graphs with normal spanning trees have been characterized by Jung \cite{jung1969wurzelbaume} and by Pitz \cites{pitz2020unified,pitz2021proof}. 
The properties `componental' and `tight' are included in \ref{itm:LinkedTD:Linked} of \cref{main:LinkedTightCompTreeDecompnew} to prohibit certain trivial, undesirable solutions for finding linked tree-decompositions into finite parts, see \cite{LinkedTDInfGraphs}. 

\cref{main:LinkedTightCompTreeDecompnew} has a number of consequences, in particular for the end structure of graphs: every graph which admits a \td\ into finite parts has a \td\ into finite parts that homeomorphically displays all the ends of $G$ and their combined degrees, resolving a question of Halin from 1977 \cite{halin1977systeme}*{\S6}. This latter \td\ yielded short, unified proofs of the characterisations due to Robertson, Seymour and Thomas of graphs without half-grid minor \cite{robertson1995excluding}*{(2.6)}, and of graphs without binary tree subdivision \cite{seymour1993binarytree}*{(1.5)}. We refer the reader to the introduction of \cite{LinkedTDInfGraphs} for details.

\medskip
Now both theorems above assert that if a graph has a tree-decomposition with a certain property (here: width), then it also has a linked such tree-decomposition. In this paper, we investigate whether this pattern also extends to other tree-decompositions. Having in mind the above-mentioned consequences for the end structure of infinite graphs, a crucial case in point would be Carmesin's theorem from 2014~\cite{carmesin2019displayingtopends}, which asserts that every graph $G$ has a tight, componental tree-decomposition of finite adhesion that displays all the undominated ends of $G$ (see also \cite{pitz2022} for a short proof). Is there also a linked such tree-decomposition in Carmesin's theorem? If so then, as above, such a tree-decomposition would homeomorphically display all the undominated ends of the underlying graph $G$ and their  degrees.

 The purpose of this paper is to answer that question in the affirmative. 
 In fact, we will show something considerably more general: 
 Whenever a graph $G$ has a tree-decomposition of finite adhesion that displays some prescribed subset of ends $\Psi$ of $G$, then there is also a linked such tree-decomposition.

\begin{mainresult} \label{main:LinkedTD:Psi}
 For any graph $G$ and any set $\Psi$ of ends of $G$ the following statements are equivalent:
\begin{enumerate}[label=\rm{(\roman*)}]
    \item \label{itm:main:Psi:i} $G$ admits a linked, tight and componental rooted \td~of finite adhesion displaying $\Psi$.
    \item \label{itm:main:Psi:ii} $G$ admits a (rooted) tree-decomposition of finite adhesion displaying $\Psi$.
    \item \label{itemGdelta} $\Psi$ is $G_\delta$ in $|G|$.
\end{enumerate}
\end{mainresult}  

A \defn{$G_{\delta}$-set} of a topological space is any set that can be written as a countable intersection of open sets. 
For example, the undominated ends form a $G_\delta$-set. To see this, fix in each component some vertex, let $U$ be their union, and consider the set $B_n(U)$ of all vertices within graph distance at most $n$ from a vertex in $U$. The closure of each $B_n(U)$ contains only ends that are dominated by a vertex in $B_n(U)$, and hence the set of undominated ends is precisely the complement of the countable union of closures of the balls $B_n(U)$, for $n \in \N$, see \cite{koloschin2023end}*{Theorem~10.1} for details.

This $G_\delta$-condition was originally isolated by Koloschin, Krill and Pitz~\cite{koloschin2023end}, who also established the equivalence of \ref{itm:main:Psi:ii} with \ref{itemGdelta}. However, we shall not rely on this result for the following reason:
It is obvious that \ref{itm:main:Psi:i} implies \ref{itm:main:Psi:ii}, and the implication \ref{itm:main:Psi:ii} implies \ref{itemGdelta} is also immediate: 
Given a set of ends $\Psi$ in a graph $G$ for which there exists a rooted \td\ $(T,\cV)$ of finite adhesion that displays $\Psi$, then considering the set $V_n$ of all vertices in $G$ which lie in a bag $V_t$ for a node $t$ on the first $n$ levels of $T$, we see that the set $\Psi$ is precisely the complement of the countable union of the closed sets $\overline{V_n}$, and hence a $G_\delta$-set. 
Thus, our argument in this paper will focus on the implication \ref{itemGdelta} implies \ref{itm:main:Psi:i}. We give a proof sketch below, but before doing so let us mention two consequences of our theorem concerning the end-structure of $\Psi$:

As indicated above, any linked \td s~of finite adhesion that is tight and componental conveys all combinatorial information about the ends in $\Psi$: their degree, their sets of dominating vertices, and hence their combined degrees, \cite{LinkedTDInfGraphs}*{Lemma~3.2 and~3.3}.
Thus, \cref{main:LinkedTD:Psi} immediately yields the following corollaries:

\begin{maincorollary} \label{maincor:DisplayingEndDegrees1}
    Let $G$ be a graph. Then for every set of ends $\Psi \subseteq \Omega(G)$ that is $G_\delta$ in $|G|$, there exists a \td\ of finite adhesion that homeomorphically displays all ends of $\Psi$, their dominating vertices, and their combined degrees.
\end{maincorollary}

Specifically, for $\Psi$ the set of undominated ends, we obtain the following sharpening of Carmesin's theorem:

\begin{maincorollary} \label{maincor:DisplayingEndDegrees2}
    Every graph $G$ admits a \td\ of finite adhesion that homeomorphically displays all the undominated ends of $G$ and their degrees, so that for every $t \in T$ and every ray $R \subset T$ starting in $t$, there is $t' \in R$ with $V_{t} \cap V_{t'} = \emptyset$. 
\end{maincorollary}

Finally, Albrechtsen, Jacobs, Knappe and Pitz \cite{ExamplesLinkedTDInfGraphs} gave a number of examples showing that the property `linked' in \cref{main:LinkedTD:Psi}~\ref{itm:main:Psi:i} can neither be strengthened to `lean' \cite{ExamplesLinkedTDInfGraphs}*{Example~1} nor to its `unrooted' version that required the `linked' property between every pair of nodes and not just comparable ones \cite{ExamplesLinkedTDInfGraphs}*{Example~2}. 
Furthermore, we cannot improve the \td\ in \cref{maincor:DisplayingEndDegrees2} to have `upwards disjoint adhesion sets' \cite{ExamplesLinkedTDInfGraphs}*{Example~5.2}.

However, one might hope to improve the \td\ from \ref{itm:main:Psi:i} of \cref{main:LinkedTD:Psi} to have \emph{connected parts}, i.e.\ for every node $t \in T$ the part $G[V_t]$ is connected. Indeed, Koloschin, Krill and Pitz \cite{koloschin2023end}*{Theorem~7.1} showed that if a set of ends $\Psi$ can be displayed by a \td\ of~$G$ of finite adhesion, then $\Psi$ can also be displayed by a \td\ of finite adhesion with connected parts.
However, it is in general not possible to additionally require connected parts in \ref{itm:main:Psi:i} of \cref{main:LinkedTD:Psi} (see \arXivOrNot{\cref{Appendix:counterexampleConnectedParts}}{the appendix of the arXiv version of this paper} for details).

\section{Sketch of the proof} \label{section:ProofSketch}

In this section we give a sketch of the proof of the remaining implication in \cref{main:LinkedTD:Psi}. Any undefined terminology can be found in the preliminaries section,  \cref{section:Background}.

Now to prove that \ref{itemGdelta} implies \ref{itm:main:Psi:i}, let $G$ be a graph, and let $\Psi$ be some non-empty set of ends of $G$ that is $G_\delta$ in $|G|$. Written out, this means that there exists a sequence $X_1 \subseteq X_2 \subseteq \ldots$ of closed subsets of $V(G) \cup \Omega(G)$ such that $\bigcup_{n \in \N} X_n = V(G) \cup \Xi$ where $\Xi := \Omega(G) \setminus \Psi$. 
\medskip

First, in order to construct a tight and componental \td\ of~$G$ displaying $\Psi$ (which need not be linked) one can follow the strategy as Koloschin, Krill and Pitz~\cite{koloschin2023end}: The idea is to construct the \td\ level-by-level, where in each step one has a \td\ $(T^i, \cV)$ of some induced subgraph $G^i \subseteq G$ that already covers $X_i$, and then extend it to a \td\ $(T^{i+1}, \cV)$ which covers $X_{i+1}$, so that $T^i \subseteq T^{i+1}$. To make the \td\ componental, one adds for every component of $G-G^i$ a new node $t_C$, and makes $t_C$ adjacent to a node of $T^i$ containing $N(C)$ (assuming we ensured in step $i$ that such a node exists). To guarantee that $(T^{i+1}, \cV)$ covers $X_{i+1}$, every bag $V_{t_C}$ associated with a newly added node $t_C$ should contain $V(X_{i+1}) \cap V(C)$. However, if we define $V_{t_C} := V(X_{i+1}) \cap V(C)$, then in the next step, the components of $G-G^{i+1}$ might have infinite neighbourhood in $G^{i+1}$, and thus we will not achieve finite adhesion.

To solve this problem, Koloschin, Krill and Pitz showed that every closed set $X$ of vertices and ends of $G$ has an \defn{envelope}: a set $X^* \supseteq V(X)$ of vertices of $G$ whose boundary contains precisely the ends in $X$, and such that every component of $G-X^*$ has only finitely many neighbours in~$X^*$ (see \cref{sec:Envelopes}).  
Then letting $V_{t_C}$ be some envelope of $V(X_{i+1}) \cap V(C)$, we obtain a tight and componental rooted \td\ $(T, \cV)$ of finite adhesion. Moreover, the closures of its bags contain precisely the ends in $\Xi$ as $\bigcup_{n \in \N} X_n = V(G) \cup \Xi$ and every bag is an envelope of a subset of some $X_i$. Hence, by \cref{lem:compDisplaysBoundary}, $(T, \cV)$ displays $\Psi$.
\medskip

While this construction already achieves most properties that we need for \cref{main:LinkedTD:Psi}, it fails to ensure linkedness. Indeed, the construction only guarantees that the adhesion sets of $(T, 
\cV)$ are finite, but it does not choose the adhesion sets in any sense minimal; in fact, all the minimal separators along a ray in $T$ could be hidden inside the bags in $\cV$. To solve this problem, we employ a (variant of an) algorithm developed by Albrechtsen, Jacobs, Knappe and Pitz~\cite{LinkedTDInfGraphs}. Essentially, given a finite set $X$ of vertices of a graph $G$, this algorithm outputs a star $\sigma$ of separations of order $<|X|$ such that $X \subseteq \interior(\sigma)$ and such that every end in $\Psi$ which can be separated from $X$ by a separation of order $<|X|$ is separated from $X$ by a separation in $\sigma$ (see \cref{section:Algorithm} for the details). 
\medskip

\begin{figure}[ht]
    \centering
    \includegraphics[width=0.7\linewidth]{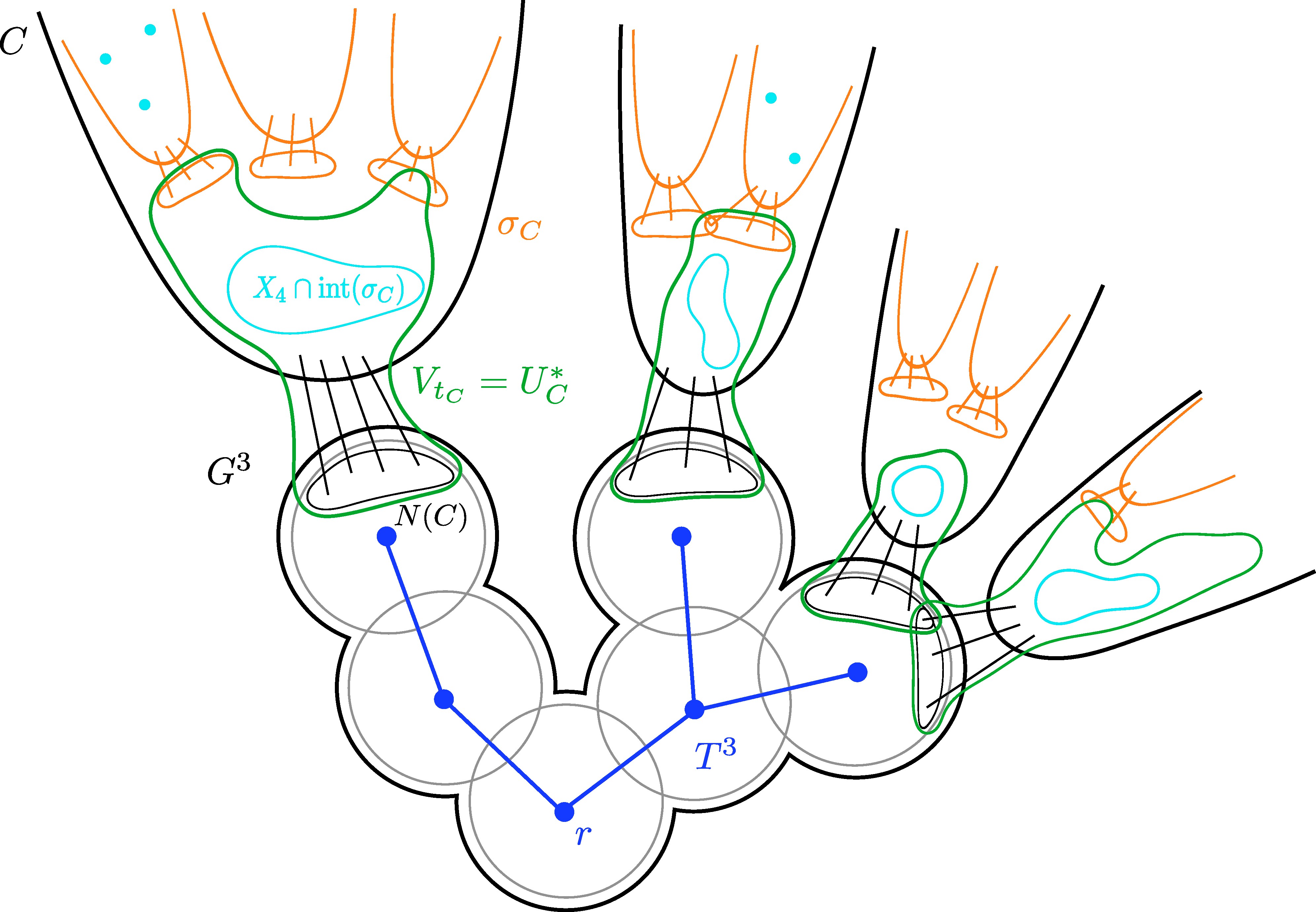}
    \caption{Depicted is a step of the recursive construction of the \td\ in the proof of \cref{main:LinkedTD:Psi}.}
    \label{fig:Sketch:TreeDecomp}
\end{figure}

To construct the desired \td\ for \cref{theorem:displayPrescribed}, we will follow the same strategy as described above, except that given $(T^i, \cV)$ and a component $C$ of $G-G^i$, we first run the aforementioned algorithm in $G[V(C) \cup N(C)]$ on input $X := N(C) \subseteq V(G^i)$ to obtain a star of finite order separations $\sigma_C$ as above; then, still working inside $G[V(C) \cup N(C)]$, we let $V_{t_C}$ be an envelope $U_C^*$ of $U_C$ together with $N(C)$ where $U_C$ is the union of $X_{i+1} \cap \interior(\sigma)$ together with the separators $A \cap B$ of all separations $(A,B) \in \sigma$ whose side $\cl(B)$ intersects $X_{i+1}$ non-emptily (see \cref{fig:Sketch:TreeDecomp}). By choosing the envelope carefully, we may arrange that $V_{t_C} = U_C^* \subseteq \interior(\sigma_C)$ (see \cref{sec:Envelopes}). This ensures 
that no separator $A \cap B$ of a separation $(A,B) \in \sigma_C$ gets hidden inside $V_{t_C}$. In the following steps, we ensure that any such separation $A \cap B$ will appear as an adhesion set in $(T, \cV)$ eventually, which will be the key to prove linkedness. See \cref{section:prescribedLinked} for details.

The arising \td\ $(T, \cV)$ is tight and componental, and it displays the ends in $\Psi$ and their combined degrees. However, $(T, \cV)$ might not yet be linked, as some edges corresponding to adhesion sets which arose from an envelope (instead of some star $\sigma_C$) might violate the linkedness property. The final step our construction is then to contract all those edges, which will yield the desired \td, which is achieved in the final page of \cref{section:prescribedLinked}.

\section{Preliminaries}\label{section:Background}

For graph-theoretical notation we closely follow \cite{bibel}.
When we refer to the natural numbers we do not include $0$. We denote the set of natural numbers including $0$ by $\mathbb{N}_0$. In what follows we recall notions that we use throughout this paper.

\subsection{Ends}
Let $G$ be a graph. A $1$-way infinite path in $G$ is called a \defn{ray}, and its $1$-way infinite subpaths are called its \defn{tails}. 
An \defn{end} of $G$ is an equivalence class of rays where two rays are equivalent if no finite set of vertices separates their underlying vertex sets. 
The set of all ends of $G$ is denoted by \defn{$\Omega(G)$}. For every end $\eps \in \Omega(G)$ and every finite set $S \subseteq V(G)$, there is a unique component of $G-S$ that contains rays from $\eps$. We denote this unique component of $G-S$ by \defn{$C(S,\eps)$} and say that $\eps$ \defn{lives} in the component $C(S, \eps)$. Moreover, we sometimes write \defn{$C_G(S,\eps)$} to emphasize that $C_G(S, \eps)$ is a component of $G-S$. This will be useful when there are multiple graphs in the picture. 

Let $X \subseteq V(G)$ be a set of vertices and $\eps$ an end of $G$. A finite set $S \subseteq V(G)$ is an \defn{$X$--$\eps$ separator} if $C(S, \eps)$ is disjoint from $X$.
A vertex $x$ of $G$ \defn{dominates} an end $\eps$ of $G$ if there exists no finite $x$--$\eps$ separator that does not include $x$. Equivalently, for every ray $R$ of $\eps$ there exists a subdivided star attached to $V(R)$ with center $x$.
We denote by \defn{$\Dom(\eps)$} the set of all vertices of $G$ that dominate $\eps$. An end of $G$ that is not dominated by some vertex is called \defn{undominated}. 

The \defn{degree} of an end is the supremum over all cardinalities of families of pairwise disjoint rays in that end. Halin \cite{halin65}*{Satz~1} sowed that this supremum is always attained. We denote the degree of an end $\eps$ by \defn{$\deg(\eps)$}. The \defn{combined degree} of an end $\eps$ is defined as the sum $\deg(\eps) + |\Dom(\eps)|$. One standard result is the following.
\begin{lemma}\label{lem:uncountableEndInfDom}\cite{PolatEME1}*{Proposition~2.11}
    Let $\eps$ be an end of uncountable degree. Then $\Dom(\eps)$ is infinite.
\end{lemma}

Let $X$ and $Y$ be two sets of vertices in a graph $G$. A set of vertices or subgraph $H$ in $G$ \defn{meets} $X$ if $H$ contains a vertex of $X$. We call a path $P$ an \defn{$X$--$Y$ path} if $P$ meets $X$ exactly in its first vertex and $Y$ exactly in its last vertex. We say $X$ is \defn{linked to} $Y$ if there exists a family $\{ P_x \colon x \in X \}$ of pairwise disjoint $X$--$Y$ paths $P_x$ such that $x \in P_x$ for every $x \in X$. 
Similarly, we define when a set $X$ is \defn{linked to} an end $\eps$ of $G$: 
An \defn{$X$--$\eps$ ray} is any ray in $\eps$ that starts in $X$, and an \defn{$X$--$\eps$ path} is any $X$--$\Dom(\eps)$ path in $G$. The set $X$ is \defn{linked to} $\eps$ if there exists a family $\{P_x \colon x \in X \}$ of pairwise disjoint $X$--$\eps$ paths and rays $P_x$ such that $x \in P_x$ for every $x \in X$. 
\subsection{Topological notions}
Let $G$ be a graph. 
Given an end $\eps \in \Omega(G)$ and a finite set $S \subseteq V(G)$, let \defn{$\Omega(S,\eps)$} be the set of all ends $\psi$ of $G$ for which $C(S, \eps) = C(S, \psi)$. We define \defn{$\hat{C}(S, \eps)$} as the union $C(S, \eps) \cup \Omega(S, \eps)$. 
The collection of all singletons $\{ v \}$ for $v \in V(G)$ together with all sets $\hat{C}(S, \eps)$ for $\eps \in \Omega(G)$ and finite $S \subseteq V(G)$ forms the basis for a topology on the space $|G| = V(G) \cup \Omega(G)$.\footnote{Often, the edges of $G$ are also included in the topological space $|G|$. For our purposes, this is not necessary; the property of being a $G_\delta$ sets of ends is independent of the inclusion of edges \cite[Lemma~2.5]{koloschin2023end}.}

Throughout this paper we only consider the above defined topology when talking about the boundary or closure of sets of vertices and ends in some graph $G$. Moreover, when we refer to a set of vertices and ends in some graph as closed or open, we mean that the set is closed or open with respect to the above defined topology in its underlying graph. We denote the closure of some set $X \subseteq V(G) \cup \Omega(G)$ with respect to our topology in some graph $G$ by \defn{$\cl_G(X)$} or \defn{$\cl(X)$} and the boundary of $X$ by \defn{$\partial_G(X)$} or \defn{$\partial(X)$} if the graph~$G$ is clear from the context. Let us recall that the boundary is defined as $\cl(X) \setminus \interior(X)$, where the interior of $X$, denoted by \defn{$\interior(X)$}, is the union of all subsets of $X$ that are open in $G$. As such, the interior of any set $X \subseteq V(G)$ equals $X$ and the boundary only ever contains ends. Moreover, it is straightforward to check that the boundary of some set $X \subseteq V(G)$ of vertices consists of exactly those ends $\eps$ of $G$ for which there exists no finite $X$--$\eps$ separator. Also, an end $\eps$ lies in the boundary of a set $X \subseteq V(G)$ of vertices if and only if there exists a \textit{comb attached to $X$ with spine in $\eps$}. A \defn{comb} $H$ is the union of a ray $R$ together with infinitely many disjoint paths, who have exactly their first vertex on $R$. We call the last vertices of these paths the \defn{teeth} of $H$ and the underlying ray the \defn{spine} of $H$. The paths are allowed to be trivial and as such the teeth are not necessary disjoint to $R$. A \defn{comb is attached} to a set $X$ if all its teeth are contained in $X$.

\subsection{Separations}
An (oriented) \defn{separation} of a graph $G$ is an ordered pair $(A,B)$ of sets of vertices $A,B \subseteq V(G)$ such that $A \cup B = V(G)$ and there is no edge in $G$ with one endpoint in $A \setminus B$ and its other in $B \setminus A$. We call $A$ and $B$ the \defn{sides} of $(A,B)$, where $A$ is the $\defn{left side}$, $B$ the $\defn{right side}$ and $A \setminus B$ the \defn{strict left side} and $B \setminus A$ the \defn{strict right side} of $(A,B)$. We refer to $A \cap B$ as the \defn{separator} of $(A,B)$ and $| A \cap B|$ as the \defn{order} of $(A,B)$. A set $S$ of separations is of \defn{finite order} if all separations in $S$ have finite order. 

A \defn{star} $\sigma$ (of separations) is a set of separations of $G$ such that for any two separations $(A,B), (C,D) \in \sigma$ it holds that $A \subseteq D$ and $B \supseteq C$. Its \defn{interior}, denoted by \defn{$\is$}, is $\is := \bigcap \{B \colon (A,B) \in \sigma\}$.
A star $\sigma$ is \defn{left-tight} if for all $(A,B) \in \sigma$ there exists a component $C$ of $G[A \setminus B]$ with $N_G(C) = A \cap B$ and \defn{left-componental} if $G[A \setminus B]$ is connected for all $(A,B) \in \sigma$.

\subsection{Tree-decompositions}

Let $T$ be a tree. We write $t \in T$ or $e \in T$ if its clear that $t$ is a node or $e$ an edge of $T$. A tree $T$ is \defn{rooted} if it contains a special node $r \in T$, its \defn{root}, denoted by \defn{$\rt(T)$}. A rooted tree $T$ has a natural partial order \defn{$<_T$} on $V(T) \cup E(T)$, its \defn{tree order}, which depends on $r = \rt(T)$: for given nodes or edges $x,y \in  V(T) \cup E(T)$, we write $\defn{xTy}$ for the unique $\subseteq$-minimal path in $T$ which contains $x$ and $y$, and we then set $x \leq_T y$ if $x \in rTy$. 

In a rooted tree, a \defn{leaf} is any maximal node in the tree-order $\leq_T$.
The \defn{height} of a given node $t \in T$ is its distance from the root. We denote the height of $t$ by \defn{$\tht(t)$}. 
When given a ray $R$, let \defn{$<_R$} be the tree order of the rooted tree $R$ rooted at its first vertex.
\smallskip

A \defn{(rooted) \td} of a graph $G$ is a pair $(T, \cV)$, where $T$ is a (rooted) tree, its \defn{decomposition tree}, and $\cV = (V_t \colon t \in T)$ a family of subsets of $V(G)$, such that 
\begin{enumerate}[label=(T\arabic*)]
    \item \label{defn:T1} $G = \bigcup_{t \in T} G[V_t]$, and
    \item \label{defn:T3} for any $v \in V(G)$, the subgraph of $T$ induced by $\{ t \in V(T) \colon v \in V_t \}$ is connected.
\end{enumerate}

The sets $V_t$ are the \defn{bags} of $(T, \cV)$, the induced subgraphs $G[V_t]$ are its \defn{parts}, and $V_e := V_x \cap V_y$, for $e=xy \in T$, are the \defn{adhesion sets}.
A \td\ has \defn{finite adhesion} if all its adhesion sets are finite.

A \td\ is essentially a decomposition of the original graph into induced subgraphs that reflects the tree-like structure of $T$: 
Given any edge $e=xy \in T$ there exist precisely two components of $T-e$. Let $A$ be the component of $T-e$ containing $x$ and $B$ the component of $T-e$ containing $y$. Setting $U^e_x := \bigcup_{t \in A} V_t$ and $U^e_y := \bigcup_{t \in B} V_t$ 
it follows that $(U^e_x, U^e_y)$ is a separation of $G$ with separator $V_e$ \cite{bibel}*{Lemma~12.3.1}. 

If $(T, \cV)$ is a rooted \td, and $e = xy \in T$ an edge with $x <_T y$, then we call $\defnm{G \up e} := G[U^e_y]$ the \defn{part above $e$} and $\defnm{G \strictup e} := G[U^e_y \setminus V_e]$ the \defn{part strictly above $e$}.

A rooted \td\ $(T, \cV)$ of a graph $G$ is 
\begin{itemize}
    \item \defn{tight} if for every edge $e \in T$ there is a component $C$ of $G\strictup e$ with $N(C) = V_e$,
    \item \defn{componental} if for every edge $e \in T$ the part $G \strictup e$ strictly above $e$ is connected, and
    \item \defn{linked} if for every pair of comparable edges $e <_T f \in T$ there are $\min \{ |V_g| \colon g \in E(eTf) \}$ pairwise disjoint $V_e$--$V_f$ paths in $G$ \cite{thomas1989wqo}.
\end{itemize}

Let $(T,\cV)$ be a \td\ of a graph $G$ of finite adhesion, and let $\eps$ be an end of~$G$. Let edge $e=xy \in T$ be an edge of $T$, and let $(U^e_x, U^e_y)$ be the separation induced by $e$. As all adhesion sets of $(T, \cV)$ are finite, there is a unique side of $(U^e_x, U^e_y)$ that includes $C(V_e, \eps)$. Directing every edge of $T$ to the unique side including $C(V_e, \eps)$ gives rise to a direction on $T$. For any node $t \in T$, there can be at most one outwards directed edge with respect to this direction, as, for any two edges $e,f$ incident with some node $t \in T$, the component $C(V_e \cup V_f, \eps)$ must choose a side.
As such, either all edges point towards a unique node of $T$ or there exists a unique rooted ray of $T$ of which all edges are directed away from the root of $T$. In the latter case, we denote this unique ray of $T$ by \defn{$R_\eps$}. We encode this behavior by a map $\defnm{f_{\cT}}: \Omega(G) \longrightarrow V(T) \cup \Omega(T)$ which maps every end of $G$ to this unique vertex or end of $T$. We call $f_{\cT}^{-1}[\Omega(T)]$ the \defn{boundary} of $(T, \cV)$.
Let $\Psi \subseteq \Omega(G)$. We say that $(T, \cV)$

\begin{itemize}
\item  \defn{displays $\Psi$} if $\Psi$ is the boundary of $(T, \cV)$ and $f_{\cT}$ restricted to $\Psi$ is a bijection between $\Psi$ and $\Omega(T)$ \cite{carmesin2019displayingtopends},
\item \defn{displays $\Psi$ homeomorphically} if $\Psi$ is the boundary of $(T, \cV)$ and $f_{\cT}$ restricted to $\Psi$ is a homeomorphism between $\Psi$ and $\Omega(T)$ \cites{koloschin2023end},
\item \defn{displays all dominating vertices of $\Psi$} if $\liminf_{e \in E(R_\eps)} V_e = \Dom(\eps)$\footnote{The set-theoretic $\liminf_{n \in \N} A_n$ consists of all points that are contained in all but finitely many $A_n$. For a ray $R= v_0e_0v_1e_1v_1 \dots$ in $T$, one gets $\liminf_{e \in E(R_\eps)} V_e = \bigcup_{n \in \N} \bigcap_{i \geq n} V_{e_i}$.} for all $\eps \in \Psi$, and
\item \defn{displays all combined degrees of $\Psi$} if $\liminf_{e \in E(R_\eps)} |V_e| = \Delta(\eps)$ for all $\eps \in \Psi$. 
\end{itemize}


Roughly, every componental, rooted \td\ homeomorphically displays all the ends \cite{koloschin2023end}*{Lemma~3.1}, every tight such \td\ displays all dominating vertices (\cite{LinkedTDInfGraphs}*{Lemma~3.2}), and every such linked tight \td\ displays all combined degrees (\cite{LinkedTDInfGraphs}*{Lemma~3.3}).

\begin{lemma}\cite{koloschin2023end}*{Lemma~3.1} \label{lem:compDisplaysBoundary}
    Let $G$ be a graph. Every componental rooted \td\ of~$G$ of finite adhesion homeomorphically displays its boundary.
\end{lemma}

\begin{lemma}\cite{LinkedTDInfGraphs}*{Lemma~3.2 \&~3.3} \label{lem:LinkedTightCompDisplaysDomAndDegrees}
    Let $G$ be a graph. Every linked, tight and componental rooted \td\ of $G$ of finite adhesion displays all dominating vertices and combined degrees of its boundary.
\end{lemma}

\subsection{Regions}
Let $G$ be a graph. A \defn{region} of $G$ is any connected induced subgraph of $G$. Let $k \in \mathbb{N}_0$. A \defn{$k$-region} of $G$ is any region $H$ of $G$ with $| N_G(H) | = k$. Moreover, a \defn{$\lk$-} or \defn{$\lek$-region} of $G$ is any $k'$-region $H$ of $G$ for some $k'$ with $k' < k$ or $k' \leq k$, respectively. Lastly, a \defn{$\lA$-region} of $G$ is any $k$-region of $G$ for $k \in \mathbb{N}_0$. All regions occuring in this paper will be $\lA$-regions.

Given an end $\eps$ of $G$, a $\lA$-region $H$ of $G$ is \defn{$\eps$-linked} if $\eps$ lives in $H$ and the neighbourhood $N_G(H)$ is linked to $\eps$. A $\lA$-region $H$ of $G$ is \defn{end-linked} if it is $\eps$-linked for some end $\eps$ of $G$.

Two regions $C$ and $D$ of $G$ \defn{touch} if they either share a vertex or there exists an edge $e$ in $G$ with one endpoint in $C$ and the other endpoint in $D$. Two regions $C$ and $D$ are \defn{nested} if they do not touch or $C \subseteq D$ or $D \subseteq C$.

\section{Envelopes} \label{sec:Envelopes}

Let $G$ be a graph. A set $X$ of vertices of $G$ has \defn{finite adhesion} in $G$ if $N(C)$ is finite for all components $C$ of $G-X$. An \defn{envelope} for a set $X \subseteq V(G) \cup \Omega(G)$ of vertices and ends is a set of vertices $X^* \supseteq X \cap V(G)$ of finite adhesion with $\partial(X^*)=\cl(X) \cap \Omega(G)$. Envelopes were originally developed by Kurkofka and Pitz \cite{kurkofka2024representationtheoremendspaces} only for subsets of $V(G)$ (see \cite{kurkofka2024representationtheoremendspaces}*{Theorem~3.2}). In this context, the second condition shortens to $\partial(X^*) = \partial(X)$. Koloschin, Krill and Pitz \cite{koloschin2023end} extended envelopes to sets of vertices and ends (see \cite{koloschin2023end}*{Theorem~5.1}). We need the following more detailed version of their theorem, which follows from \cite{koloschin2023end}*{Theorem~5.1 \& Claim~5.3}:

\begin{theorem} \label{thm:Envelope_lemma:Detailed}
    Any set $X$ consisting of vertices and ends in a graph $G$ admits an envelope $X^*$. 
    
    Moreover, one may choose $X^*$ so that for every 
    $(<\aleph_0)$-region $D$ of $G$ with $\cl(D) \cap X = \emptyset$, we have $V(D) \cap X^*$ is finite.
\end{theorem}

We need the `moreover'-part of \cref{thm:Envelope_lemma:Detailed} to conclude the following variant of \cref{thm:Envelope_lemma:Detailed}, which ensures that we can choose $X^*$ to be a subset of $\interior(\sigma_C)$ (see the proof sketch in \cref{section:ProofSketch}).

\begin{lemma} \label{lem:Envelopes:Final}
    Let $G$ be a graph, $X \subseteq V(G) \cup \Omega(G)$, and let $\cD$ be a set of pairwise non-touching  $(<\aleph_0)$-regions of $G$ whose closures avoid $X$. Then there is an envelope $X^*$ of $X$ that avoids $\bigcup \cD$.
\end{lemma}

\begin{proof}
    By applying \cref{thm:Envelope_lemma:Detailed} to $X$ we obtain an envelope $X^{**}$ of $X$. Let $\cD' \subseteq \cD$ be the set of all $D \in \cD$ such that $V(D) \cap X^{**} \neq \emptyset$, and let
    \[
    X^* := \Big(X^{**}\setminus \Big(\bigcup \cD'\Big)\Big) \cup \bigcup_{D \in \cD'} N_G(D).
    \]
    Since the regions in $\cD$ are pairwise non-touching, it clearly follows that $X^*$ avoids $\bigcup \cD$. We claim that $X^*$ is as desired. 
    \smallskip

    We first check that $\partial_G(X^*) = \partial_G(X^{**})$, which then yields $\partial_G(X^*) = \cl(X) \cap \Omega(G)$ since~$X^{**}$ is an envelope of~$X$. For this, let $\eps \in \partial_G(X^{**})$, and let~$C$ be a comb in~$G$ with spine in~$\eps$ and teeth in~$X^{**}$. If $C$ meets $X^*$ infinitely often, then $\eps \in \partial_G(X^*)$ as desired. Otherwise, $C$ contains infinitely many vertices of $\bigcup \cD'$. By the `moreover'-part of \cref{thm:Envelope_lemma:Detailed}, for every component $D \in \cD'$, the set $V(D) \cap X^{**}$ is finite. Hence, $C$ meets the neighbourhoods~$N_G(D)$ of $D \in \cD'$ infinitely often, and thus $C$ meets $X^*$ infinitely often. Hence, $\eps \in \partial_G(X^*)$. 
    \smallskip

    Conversely, let $\eps \in \partial_G(X^*)$, and let~$C$ be a comb in $G$ with spine in $\eps$ and teeth in $X^*$. If~$C$ meets~$X^{**}$ infinitely often, then $\eps \in \partial_G(X^{**})$ as desired. Otherwise, as $N_G(D)$ is finite for every $D \in \cD'$, the comb $C$ contains distinct vertices $v_1, v_2, \dots$ such that $v_i \in N_G(D_i)$ for distinct $D_i \in \cD'$. Since every $D_i$ lies in $\cD'$, the set $V(D_i) \cap X^{**}$ is non-empty, so we may pick some $u_i \in V(D_i) \cap X^{**}$.
    As every~$D_i$ is connected and $C$ meets $N_G(D_i)$, we can easily extend $C$ to a comb with teeth~$u_i$ and with the same spine as~$C$. As $u_i \in X^{**}$, it follows that $\eps \in \partial_G(X^{**})$.
    \smallskip

    It remains to check that $X^*$ has finite adhesion. For this, let $C$ be a component of $G-X^*$. If $C \cap D \neq \emptyset$ for some component $D \in \cD'$, then $C = D$ since $N_G(D) \subseteq X^*$ and $X^* \cap V(D) = \emptyset$, and hence $N_G(C) = N_G(D)$ is finite by the assumption on $\cD$. Thus, we may assume that $C$ avoids $\bigcup \cD'$, which implies that $C$ is included in a component $C'$ of $G-X^{**}$. Then 
    \[
    N_G(C) \subseteq (N_G(C') \cup V(C')) \cap X^* \subseteq N_G(C') \cup \Big(V(C') \cap \bigcup \{N_G(D) : D \in \cD'\}\Big).
    \]
    Since $N_G(C')$ is finite as $X^{**}$ has finite adhesion, and because every $N_G(D)$ is finite by the assumption on $\cD$, it suffices to show that there are at most finitely many $D \in \cD'$ with $V(C') \cap N_G(D) \neq \emptyset$. Let such $D$ be given, and let $v \in V(C') \cap N_G(D)$. 
    Since $D$ is connected and $v \in N_G(D)$, and because $D \in \cD'$ and hence $V(D) \cap X^{**} \neq \emptyset$, there is a $v$--$(V(D) \cap X^{**})$ path~$P$ in $G[V(D) \cup N_G(D)]$. Then~$P$ avoids~$X^{**}$ except in its endvertex $u$ in $X^{**}$, and hence $P$ is internally contained in $C'$. It follows that $u \in N_G(C')$. Hence, $N_G(C')$ contains at least one vertex of every component $D \in \cD'$ with $V(C') \cap N_G(D) \neq \emptyset$.
    As $N_G(C')$ is finite and the regions $D \in \cD'$ are disjoint, it follows that $C'$ intersects the neighbourhoods of at most finitely many $D \in \cD'$. This concludes the proof.
\end{proof}

\section{Separator algorithm}\label{section:Algorithm}

As described in the proof sketch in \cref{section:ProofSketch}, we will employ a (variant of an) algorithm developed by Albrechtsen, Jacobs, Knappe and Pitz~\cite{LinkedTDInfGraphs}. In this section, we introduce this algorithm. 
\smallskip

The following algorithm is taken from \cite{LinkedTDInfGraphs}*{Algorithm~7.2}, except that our algorithm already terminates after Case A of \cite{LinkedTDInfGraphs}*{Algorithm~7.2}.
\begin{algo} \label{algo:RegionAlgo}
    \textbf{Input:} a connected graph $H$; a finite set~$X$ of $k \in \N$ vertices of $H$; a set $\cD$ of pairwise non-touching end-linked $\lk$-regions that are disjoint from~$X$.

    \textbf{Output:} a transfinite sequence $\left(C_i \colon i < \kappa\right)$ indexed by some ordinal $\kappa< |H|^+$ (the successor cardinal of $|H|$), of distinct end-linked $\lk$-regions that are disjoint from $X$,  pairwise nested, and also nested with all $D \in \cD$. 
    
   \textbf{Recursion:} 
Iterate the following step for as long as possible:

\begin{quote}
As long as there is a $\lk$-region $C_i$ of $H$ that is
\begin{itemize}
    \item disjoint from $X$,
    \item not strictly contained in any $D \in \cD$, 
    \item $\eps$-linked for some end $\eps$ which lives in no~$C_j$ for $j<i$, and 
    \item nested with the $C_j$ for $j < i$ and with all $D \in \cD$, 
\end{itemize} 
then choose a \defn{nicest} such region $C_i$.
        Here, nicest\footnote{Note that there might be several nicest regions.} means that
        \begin{enumerate}[labelindent=1.5em, label=(N\arabic*), leftmargin=!]
            \item \label{nicest:ell_iminimum} $C_i$ is such an $\ell_i$-region where $\ell_i \in \N$ is minimum among such regions, and
            \item \label{nicest:inclusion-wisemaximal} $C_i$ is an inclusion-wise maximal such region subject to \cref{nicest:ell_iminimum}.
        \end{enumerate}
        Otherwise, if there is no such region, then terminate the recursion.
  \end{quote}
\end{algo}

We remark that if $R_1 \subseteq R_2 \subseteq \dots$ are $k$-regions of a graph $G$ that are all candidates for $C_i$ at some step $i$, then by \cite{LinkedTDInfGraphs}*{Lemma~7.3} their union $\bigcup_{i \in \N} R_i$ is also a candidate for $C_i$, and hence there always exists a region satisfying \ref{nicest:inclusion-wisemaximal}.
\smallskip

Roughly speaking, \cref{algo:RegionAlgo} chooses for every end~$\eps$ of~$H$ to which $X$ is not linked (and that lives in no $D \in \cD$) a region~$C$ of $H$ that is $\eps$-linked and whose neighbourhood $N(C)$ is minimal among all $\eps$-linked regions, which in particular implies that $N(C)$ is linked to~$X$ (see \cref{observation:algoLinkedEnds} below). 
\cref{algo:RegionAlgo} picks these regions one after another; first choosing such a region for every end of~$H$ that can be separated from~$X$ by a single vertex of~$H$, then choosing a region for every end of~$H$ which can be separated from $X$ by two vertices, and so on. 

Since \cref{algo:RegionAlgo} only selects nicest regions, a region~$C_i$ is never contained in a region~$C_j$ chosen at an earlier step $j < i$, and some~$C_i$ can only be a superset of some $C_j$ if $\ell_i > \ell_j$. In particular, any chain of the form $C_{i_1} \subsetneq C_{i_2} \subseteq C_{i_3} \subsetneq \ldots$ is finite (as $\ell_i < k$ for all $i < \kappa$). Therefore, every region~$C_i$ is contained in an $\subseteq$-maximal region of $(C_i : i < \kappa)$, and every component of $G-\bigcup_{i<\kappa}C_i$ is some~$C_i$.

For the proof of \cref{main:LinkedTD:Psi}, we need the following observation about \cref{algo:RegionAlgo}, which we already mentioned above:

\begin{observation} \label{observation:algoLinkedEnds}
    Given the setting of \cref{algo:RegionAlgo} and any end $\eps$ of $H$ it holds that
    \begin{enumerate}

    \item \label{property:everyEndInSomeC_i} if $\eps$ is not living in a region in $\cD$ and $X$ is not linked to $\eps$, then some $C_i$ is $\eps$-linked, and 

    \item \label{property:linkedToX} the neighbourhood $N(C_i)$ of every region $C_i$ is linked to $X$.
    \end{enumerate}
\end{observation}

For the proof of \cref{observation:algoLinkedEnds} we need two lemmas from \cite{LinkedTDInfGraphs}*{Section~7}.
The first lemma is an extension of \cite{LinkedTDInfGraphs}*{Lemma~7.4}, and yields an $\eps$-linked $\lk$-region for every end $\eps$ of $H$ to which $X$ is not linked. Note that this region might not yet be a candidate for a region $C_i$ in \cref{algo:RegionAlgo} as it might not be nested with regions picked previously by \cref{algo:RegionAlgo}.

\begin{lemma} \label{lem:existenceofregion}
    Let $H$ be a graph, $X \subseteq V(H)$ a finite set of vertices and $\eps$ an end of $H$.
    Then for every $\size$-minimal $X$--$\eps$ separator $S$ the region $C(S, \eps)$ is $\eps$-linked and disjoint from $X$.
\end{lemma}
\begin{proof}
    If $\eps$ has countable combined degree, this follows from \cite{LinkedTDInfGraphs}*{Lemma~7.4}.
    If $\eps$ has uncountable combined degree, then $\eps$ is infinitely dominated by \cref{lem:uncountableEndInfDom}. Given any finite set $A \subseteq V(H)$ of vertices, the graph $C(A,\eps)$ contains infinitely many vertices dominating $\eps$ and so any finite $X$--$\Dom(\eps)$ separator is a finite $X$--$\eps$ separator and vice versa. Let $S$ be a $\size$-minimal $X$--$\eps$ separator. By Menger's Theorem, $S$ is linked to $\Dom(\eps)$ and hence linked to $\eps$. So, $C(S, \eps)$ is an $\eps$-linked region of $H$ and disjoint to $X$.
\end{proof}

The following lemma is the key in transforming the end-linked regions provided by \cref{lem:existenceofregion} into end-linked regions that are nested with all previously picked region. 

\begin{lemma}\cite{LinkedTDInfGraphs}*{Lemma~7.6}\label{lem:theminiuncrossinglemma}
    Let $H$ be a graph and let $\cE$ be a set of pairwise non-touching end-linked $\lA$-regions.
    Suppose further that $C$ is a $k$-region with $k \in \N$ that is $\eps$-linked for some end $\eps$ that does not live in any $D \in \cE$ which is not contained in~$C$. 
    
    Then there exists an $\eps$-linked $\lek$-region $C'$ that is nested with all $D \in \cE$.
\end{lemma}
\begin{proof}
    In \cite{LinkedTDInfGraphs}*{Lemma~7.6}, the authors require the weaker condition of $\cE$ being a set of pairwise non-touching well-linked $\lA$-regions. We will not define well-linked, but note that by \cite{LinkedTDInfGraphs}*{Lemma~7.5} any end-linked region is well-linked. 
\end{proof}

Now we are in a position to prove \cref{observation:algoLinkedEnds}.

\begin{proof}[Proof of \cref{observation:algoLinkedEnds}]
    Let $(C_i \colon i<\kappa)$ be the regions picked by \cref{algo:RegionAlgo}. 
    Given $n \in \mathbb{N}$ with $n<|X|$, let $\cC^{<n}_{\max}$ be the set of all $\subseteq$-maximal regions among all $(<n)$-regions $C_i$. 

    \cref{property:everyEndInSomeC_i}:
    Suppose for a contradiction that $\eps$ does not live in any region in $\cD$, and that no $C_i$ is $\eps$-linked.
    Let $S$ be a $\size$-minimal $X$--$\eps$-separator. By \cref{lem:existenceofregion} the region $C(S,\eps)$ is $\eps$-linked and disjoint from $X$. As $X$ is not linked to $\eps$ it follows that $|S|<|X|$.
    If $\eps$ lives in a region $C_j$ for $j < \kappa$, let $i < \kappa$ be minimal with $\eps$ living in $C_i$. As $C_i$ is not $\eps$-linked by assumption, there is a $N(C_i)$--$\eps$ separator of size $<\ell_i$. Moreover, any such separator also separates $X$ from $\eps$ and as such $|S| < \ell_i$. Let $m=\ell_i$ or $m=|X|$ if $\eps$ lives in no region $C_j$ for $j < \kappa$. 
    In particular, $\eps$ does not live in any region in $\cD \cup \cC^{<m}_{\max}$. 
    Applying \cref{lem:theminiuncrossinglemma} on $H$, $\cE$ the set of all $\subseteq$-maximal regions in $\cD \cup \cC^{<m}_{\max}$ and $C := C(S,\eps)$ returns an $\eps$-linked $(<m)$-region $C'$ nested with all regions in $\cD$ and $C_i$ with $\ell_i < m$. Hence, $C'$ would have been a better choice at the step where \cref{algo:RegionAlgo} chose the first $m$-region, a contradiction.

    \cref{property:linkedToX}:
    Without loss of generality, we may assume that $C_i$ is an $\eps$-linked region. Then $\eps$ lives in no region in the set $\cC_{\max}^{<\ell_i}$. If $N(C_i)$ is not linked to $X$ then by Menger's Theorem there is an $X$--$N(C_i)$ separator $S$ of size $|S| < \ell_i$. By minimality, $S$ is disjoint to $C_i$ and hence $C(S,\eps)$ includes $C_i$ as a subgraph. Moreover, every region in $\cD$ in which $\eps$ lives is included in $C_i$ and hence in $C(S, \eps)$. Applying \cref{lem:theminiuncrossinglemma} on $H$, $\cE$ the set of all $\subseteq$-maximal regions in  $\cD \cup \cC^{<\ell_i}_{\max}$ and $C := C_H(S,\eps)$ returns an $\eps$-linked $(<\ell_i)$-region nested with all $D \in \cD$ and $C_j$ with $\ell_j < \ell_i$, contradicting the choice of $C_i$.
\end{proof}

\section{Proof of the main theorem} \label{section:prescribedLinked}

In this section we prove \cref{main:LinkedTD:Psi}. 
Since \ref{itm:main:Psi:i} implies \ref{itm:main:Psi:ii} is trivial and \ref{itm:main:Psi:ii} implies \ref{itemGdelta} is easy (cf.\ \cref{sec:Intro}), it suffices to show that \ref{itemGdelta} implies \ref{itm:main:Psi:i}, that is

\begin{theorem} \label{thm:LinkedTD:Psi:Copy}
    Let $\Psi \subseteq \Omega(G)$ be a $G_{\delta}$-set of ends in some graph $|G|$. Then $G$ admits a linked, tight and componental rooted \td~of finite adhesion displaying $\Psi$.
\end{theorem}   

As described in the proof sketch in \cref{section:ProofSketch}, we construct the desired \td\ recursively, where in each step, we have a \td\ $(T^i, \cV)$ of some subgraph $G^i$ of $G$, which we then extend into the components of $G-G^i$. For this, we add for every component $C$ of $G-G^i$ a new node $t_C$ to $T^i$, run \cref{algo:RegionAlgo} in $G[V(C) \cup N(C)]$ on input $X:=N(C)$, and then choose a suitable bag $V_{t_C}$ in $(N(C) \cup V(C)) \setminus (\bigcup \cC)$ where $\cC$ is the output of \cref{algo:RegionAlgo}.
In later steps, we will ensure that all regions in $\cC$ appear as $G\strictup f$ of some edge $f$ of $T$ above $t_C$ (cf.\ \ref{Lproperty:pathLikeRegions} of \cref{theorem:displayPrescribed} below).

The output of this procedure can be summarized as follows (note that the sets $\cD_e$ are essentially outputs $\cC$ of \cref{algo:RegionAlgo} as described above, and property \ref{item:EpsLinkedRegions} and \ref{item:LinkedNhoods} are \ref{property:everyEndInSomeC_i} and \ref{property:linkedToX} from \cref{observation:algoLinkedEnds}, which follow from \cref{algo:RegionAlgo}):

\begin{lemma} \label{theorem:displayPrescribed}
     Let $\Psi \subseteq \Omega(G)$ be any $G_{\delta}$-set of ends in some graph $G$. 
     Then $G$ admits a tight and componental rooted \td\ $(T, \cV)$ of finite adhesion displaying $\Psi$ together with a collection $(\cD_e)_{e \in E(T)}$ of sets $\cD_e$ of $(<|V_e|)$-regions in $G\up e$ such that for every edge $e \in E(T)$:
    \begin{enumerate}[label=\rm{(\arabic*)}]
    \item \label{item:EpsLinkedRegions} if an end $\eps$ living in $G\up e$ is not linked to $V_e$, then some region in $\cD_e$ is $\eps$-linked,
    \item \label{item:LinkedNhoods} the neighbourhood $N(D)$ of every region $D \in \cD_e$ is linked to $V_e$, and
    \item \label{Lproperty:pathLikeRegions} for every $D \in \cD_e$ there is a (unique) edge $f \in T$, with $e <_T f$, such that $D = G \strictup f$.
\end{enumerate}
\end{lemma}

\begin{proof}
    We may assume that $G$ is connected and that $\Psi$ is non-empty. 
    Since $\Psi \subseteq \Omega(G)$ is a $G_\delta$-set, and hence a countable intersection of open sets in $V(G) \cup \Omega(G)$, its complement is a countable union of closed sets. This means that there exists a sequence $X_1 \subseteq X_2 \subseteq X_3 \dots$ of closed subsets of $V(G) \cup \Omega(G)$ such that $\bigcup_{n \in \mathbb{N}} X_n = V(G) \cup \Xi$ where  $\Xi := \Omega(G) \setminus \Psi$ is the set of ends outside of $\Psi$.

    We construct the \td\ $(T, \cV)$ recursively.
    More precisely, we recursively choose induced subgraphs $G^1 \subset G^2 \subset \dots$ of $G$ together with \td s $(T^i, \cV^i)$ of $G^i$ of finite adhesion which all extend each other in the sense that $T^1 \subseteq T^2 \subseteq \dots$ and $V_t^i = V_t^j$ for all $i \leq j$ and $t \in T^i$.  In particular, from now on we simply write $V_t$ for the part corresponding to a vertex $t \in T^i$. The desired \td\ will then be the limit of these \td s, i.e.\ $T := \bigcup_{i \in \N} T^i$ and $\cV=\{V_t \colon t \in T\}$. Note that we will show $G = \bigcup_{i \in \N} G^i$, so $(T, \cV)$ will be a \td\ of $G$.
    
    Additionally, for every $i \in \N$, we choose a collection $(\cD_e)_{e \in E(T^i)}$ of end-linked $\lA$-regions of $G$. These collections $\cD_e$ will essentially be outputs of \cref{algo:RegionAlgo} inside $G \up e$ on input $V_e$, where these notions are taken with respect to the limit $(T, \cV)$ of the \td s $(T^i, \cV)$. 
    
    In what follows, we first make this construction precise, and then show that the arising \td\ $(T, \cV)$ is as desired.
    During the construction, we will ensure in each step $i \in \N$ that the following conditions hold:
    
    \begin{enumerate}
        \item \label{theorem:displayPrescribed:item:finiteAdhesion} $G^i$ has finite adhesion in $G$,
        \item \label{theorem:displayPrescribed:item:componental} for every component $C$ of $G-G^i$, its neighbourhood $N(C)$ is contained in a bag $V_{l_C}$ for a unique node $l_C \in T^i$ at height $i$, and 
        \item \label{theorem:displayPrescribed:item:outputAlgo} if $f$ is the unique edge below $l_C$, the set $\cD_f$ is a potential output of \cref{algo:RegionAlgo} inside the graph $G[V(C) \cup N(C)]$ on input $N(C)$.
    \end{enumerate}
    
    We start the construction by choosing an arbitrary vertex $x \in V(G)$, and letting $G^1 := (\{x\}, \emptyset)$ be the graph on the single vertex $x$. Further, we let $T^1$ be the unique rooted tree on a single node $r$ and set $V_r := \{x \}$ and $\cD_r := \emptyset$. 
    
    Now let $i > 1$ be given, and assume we have already constructed a \td~$(T^i, \cV)$ of $G^i$ together with a collection $(\cD_e)_{e \in E(T^i)}$ of end-linked $\lA$~regions satisfying \cref{theorem:displayPrescribed:item:finiteAdhesion}, \cref{theorem:displayPrescribed:item:componental} and \cref{theorem:displayPrescribed:item:outputAlgo}.
    We let \defn{$T^{i+1}$} be the tree obtained from $T^i$ by adding, for every leaf $l$ of $T^i$ at height $i$ and for every component $C$ of $G-G^i$ with $N(C) \subseteq V_l$, a new node $t_C$ neighbouring $l$. Note that by \cref{theorem:displayPrescribed:item:componental}, we add for every component $C$ of $G-G^i$ a (unique) leaf to $T^{i+1}$. 
    Essentially, $C$ will equal $G \strictup f$ in the limit of our \td s, where $f$ is the unique edge below $t_C$.
    \smallskip

    We now choose the bags for all newly added leaves of $T^{i+1}$. For this, let $C$ be some component of $G-G^i$, and let $t_C$ be the newly added leaf for $C$. Further, let $f$ be the unique edge below $t_C$. If $i > 2$, let $e$ be the unique edge below $f$, and define $\cD'_f := \{ D \in \cD_e \colon D \subsetneq C \}$. Further, let $\cD'_{f, \max}$ be the set of $\subseteq$-maximal regions among all regions in $\cD'_f$.
    
    We run \cref{algo:RegionAlgo} inside the graph $G[V(C) \cup N(C)]$ on input $X:= N(C)$, which is finite by \ref{theorem:displayPrescribed:item:finiteAdhesion}, and with $\cD$ $:= \cD'_{f, \max}$, if $i > 2$, or $\cD := \emptyset$ if $i = 2$. 
    For this, we need to check that the set $\cD'_{f, \max}$ is a valid input for \cref{algo:RegionAlgo}, i.e.\ $\cD'_{f, \max}$ is a collection of pairwise non-touching end-linked $(<|N(C)|)$-regions that are disjoint from $N(C)$. By \cref{theorem:displayPrescribed:item:outputAlgo}, $\cD'_{f, \max}$ is a collection of pairwise non-touching end-linked $(<\aleph_0)$-regions, and they are disjoint from $N(C)$ as they are contained in $C$ by definition. Moreover, since $C \in \cD_e$ and because every region in $\cD_e$ got picked $\subseteq$-maximal (by \cref{nicest:inclusion-wisemaximal}), the size of neighbourhoods of regions in $\cD'_{f, \max}$ must be strictly less than $|N(C)|$.
    
    Let us denote the output by $\cC_{f}$ and the $\subseteq$-maximal regions among all regions in $\cC_{f}$ by $\cC_{f, \max}$. Let $I_{t_C} := (V(C) \cup N(C)) \setminus \bigcup \cC_f$, and let (see also \cref{fig:prescribedConstruction})
    
    \begin{itemize}
        \item[$\bullet$] $U_{t_C,1} := X_{i} \cap I_{t_C}$,
        \item[$\bullet$] $U_{t_C,2} := X_{i} \cap \partial_G(I_{t_C})$, and
        \item[$\bullet$] $U_{t_C,3} := \bigcup \{ N(D) \colon D \in \cC_{f,\max}, X_{i} \cap \cl(D) \neq \emptyset \}$.
    \end{itemize}

    \begin{figure}[htbp]
    \centering
    \includegraphics[scale=0.55]{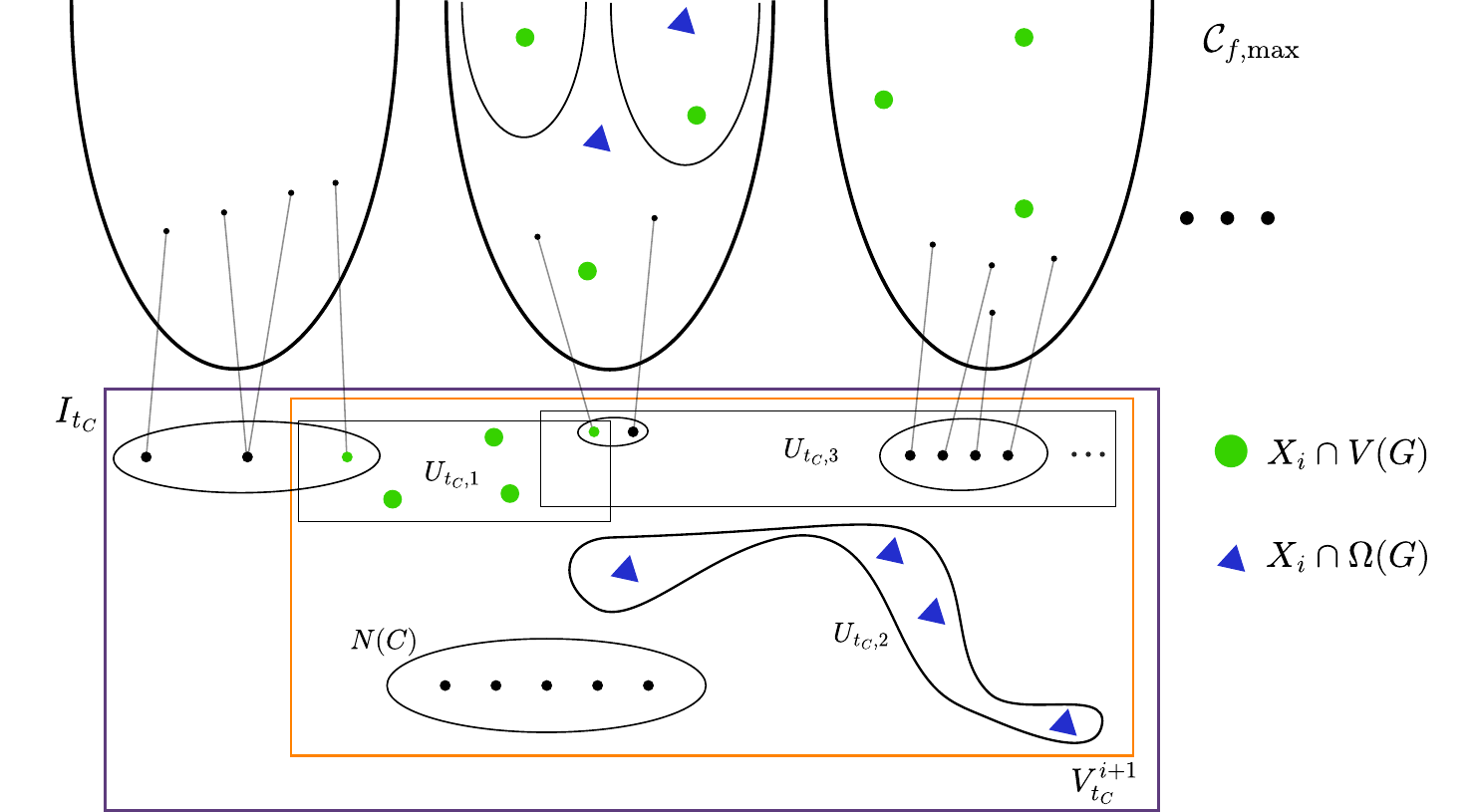}
    \caption{Depicted is the graph $G[V(C) \cup N(C)]$. The components are maximal components chosen by \cref{algo:RegionAlgo}. The orange subgraph is an envelope of $U_{t_{C,1}} \cup U_{t_{C_,2}} \cup U_{t_{C,3}} \cup N(C)$. This will be the bag associated with $t_C$.}
    \label{fig:prescribedConstruction}
    \end{figure}

   Applying \cref{lem:Envelopes:Final} 
   now yields our bag for $t_C$: 
   We let $\defnm{V^{i+1}_{t_C}} := (U_{t_C,1} \cup U_{t_C,2} \cup U_{t_C,3} \cup N(C))^*$ be some envelope in $G[V(C) \cup N(C)]$ that avoids $\bigcup \cD'_f \cup \bigcup \cC_f$ (which exists by \cref{lem:Envelopes:Final}). 
   
   Lastly, we define $\defnm{\cD_{f}} := \cD'_f \cup \cC_f$, and after having defined all new bags we let $\defnm{G^{i+1}} := G[\bigcup_{t \in T^{i+1}} V_t]$. This completes the construction of $(T^{i+1}, \cV^{i+1})$.
   \smallskip
   
   Next, we verify that $(T^{i+1}, \cV^{i+1})$ satisfies \cref{theorem:displayPrescribed:item:finiteAdhesion}, \cref{theorem:displayPrescribed:item:componental} and \cref{theorem:displayPrescribed:item:outputAlgo}. For this, let $C'$ be any component of $G-G^{i+1}$. Since $G^i \subseteq G^{i+1}$, there exists a unique component $C$ of $G-G^i$ with $C' \subseteq C$. Let further $t_C \in T^{i+1}$ be the (unique) node at height $i+1$ that we added to $T^{i+1}$ for $C$. Then $N(C') \subseteq V_{t_C}$ as $N(C) \subseteq V_{t_C}$ and because $V_{t_{\tilde{C}}} \subseteq \tilde{C}$ for every component $\tilde{C}$ of $G^i$. Thus, $(T^{i+1}, \cV^{i+1})$ satisfies \ref{theorem:displayPrescribed:item:componental}. Since $V_{t_C}$ has finite adhesion inside the graph $G[V(C) \cup N(C)]$, the neighbourhood $N(C')$ of $C'$ is finite. Hence, $(T^{i+1}, \cV^{i+1})$ also satisfies \cref{theorem:displayPrescribed:item:finiteAdhesion}. For \ref{theorem:displayPrescribed:item:outputAlgo}, let $f$ be the unique edge below $t_C$. We have to show that

   \begin{claim} \label{claim:regionsAlgo}
    The set $\cD_f$ is a potential output of \cref{algo:RegionAlgo} inside the graph $G[V(C) \cup N(C)]$ on input $X:= N(C)$ (and $\cD := \emptyset$).  
\end{claim}
\begin{proof}\renewcommand{\qedsymbol}{$\blacksquare$}
    If $i=2$, the claim holds by definition. If $i>2$, let $e=sl_C$ be the unique edge below $f$ and $C_{s}$ the unique component of $G-G^{i-1}$ corresponding to $s$. By \cref{theorem:displayPrescribed:item:outputAlgo}, the collection $\cD_e$ is a potential output of \cref{algo:RegionAlgo} inside the graph $G[V(C_s) \cup N(C_s)]$ on input $N(C_s)$ (and $\cD = \emptyset$). Hence, $\cD_e=(A_i \colon i< \kappa_1)$ where the $A_i$ are enumerated in the order they have been picked by \cref{algo:RegionAlgo}. Moreover, by construction, $\cC_f=(B_i \colon i< \kappa_2)$ is an output of \cref{algo:RegionAlgo} inside the graph $G[V(C) \cup N(C)]$ on input $X:=N(C)$ and $\cD:=\cD'_{f, \max}$. 
    
    We claim that the concatenation of first picking every region in $\cD'_f = \{ D \in \cD_e \colon D \subsetneq C \}$ in the order induced by $\cD_e = (A_i : i < \kappa_1)$, and then picking every region in $\cC_f$ in the order $(B_i : i < \kappa_2)$, yields a potential output of \cref{algo:RegionAlgo} in the graph $G[V(C) \cup N(C)]$ on input $N(C)$ (and $\cD=\emptyset$). 

    Indeed the $A_i$'s are nested, disjoint from $N(C)$, and end-linked, and each of them is a nicest such region because $\cD_e$ was a potential output of \cref{algo:RegionAlgo} on input $N(C_s)$. In particular, by \cref{nicest:inclusion-wisemaximal} and because every region in $\cD_e$ is nested with $C$, the $A_i$'s (which are strictly contained in $C$) are $(<|N(C)|)$-regions. So the $A_i$'s are also nicest regions (at step $i$, respectively) when running \cref{algo:RegionAlgo} inside $G[V(C) \cup N(C)]$ on input $N(C)$. Hence, $(A_i : i < \kappa_1)$ is an initial sequence (up to $\kappa_1$) of a potential output of \cref{algo:RegionAlgo} inside $G[V(C) \cup N(C)]$ on input $N(C)$ (and $\cD = \emptyset)$. 
    As $\cC_f$ is an output of \cref{algo:RegionAlgo} inside $G[V(C) \cup N(C)]$ on input $N(C)$ and $\cD'_{f, \max}$, which contains precisely all the maximal $A_i$, every $B_i$ is a nicest region (in the respective step) in a run of \cref{algo:RegionAlgo} (with $\cD = \emptyset$) which first picked $(A_i : i < \kappa_1)$.
    Hence, the concatenation is a valid output.
    \end{proof}
    It follows that $(T^{i+1}, \cV^{i+1})$ satisfies \cref{theorem:displayPrescribed:item:finiteAdhesion}, \cref{theorem:displayPrescribed:item:componental} and \cref{theorem:displayPrescribed:item:outputAlgo}.
    \smallskip

    Let $(T, \cV)$ be the limit of the $(T^i, \cV)$'s, i.e.\ $T := \bigcup_{i \in \N} T^i$ and $V_t := V_t^i$ for every node $t \in T$ where $i \in \N$ is such that $t \in T^i$. Note that $(T, \cV)$ is a \td\ of $\bigcup_{i \in \N} G^i$.
    Further, let $(\cD_e)_{e \in E(T)}$ be the collection of all $\cD_e$ for edges $e \in T$, where we note that during the construction of $(T, \cV)$ we chose for every edge $e \in T$ a unique collection $\cD_e$.
    We claim that $(T, \cV)$ and $(\cD_e)_{e \in E(T)}$ are as desired. In what follows, we check all required properties one by one.

    First, we show that $(T, \cV)$ is a \td\ of $G$, for which it suffices to prove that

\begin{claim} \label{lem:displayPrescribed:limitIsG}
        $\bigcup_{i \in \N} G^i = G$.
    \end{claim}
    \begin{proof}\renewcommand{\qedsymbol}{$\blacksquare$}
     Since all $G^i$ are induced subgraphs of $G$, it suffices to show that every vertex of $G$ is contained in some $G^i$. 
     Let $x \in V(G)$ be an arbitrary vertex, and let $n \in \mathbb{N}$ be large enough such that $x \in X_n$. If $x \in V(G^{n})$ then we are done. 
     
     Otherwise, let $C$ be the unique component of $G-G^{n-1}$ that contains $x$. By construction, we added a new node $t_C$ to $T^{n}$ with $U_{t_C,1} \cup U_{t_C,3} \subseteq V_{t_C}$. 
     Since $x$ is not contained in $G^{n}$, and hence not in $V_{t_C}$, it is in particular not an element of $U_{t_C,1}$. As $x \in X_n$, it follows that $x \notin I_{t_C}$, and hence $x$ is contained in a (unique) region $D \in \cC_{e,\max}$ where $e$ is the unique edge below $t_C$. 
     Again since $x \in X_n$, the neighbourhood $N(D)$ of $D$ is a subset of $U_{t_C,3}$, and hence, as $V_{t_C}$ avoids $\cC_e \ni D$, the region $D$ is a component of $G-G^{n}$. 
     Let $g \in T^{n+1}$ be the (unique) edge incident with $t_C$ that we added to $T^{n+1}$ for $D$ in step $n+1$. Then $|V_g|=|N(D)|<|V_e|$ by \cref{theorem:displayPrescribed:item:outputAlgo}. Thus, by repeating the above argument, it follows that after a finite number of steps $x$ must be contained in $G^m$ for some $m \in \N$ with $n < m \leq n+|V_e|-1$. 
\end{proof}

Next, we check that

\begin{claim} \label{claim:TightAndComponental}
$(T, \cV)$ is tight and componental and has finite adhesion. 
\end{claim}

\begin{proof}\renewcommand{\qedsymbol}{$\blacksquare$}
    Let $e \in E(T)$ be any edge of $T$. Let $i$ be the step in which we extended $T^i$ by the edge~$e$, and let $t$ be the (unique) node of $T$ at height $i+1$ that is incident with $e$. By construction $t = t_C$ for a (unique) component $C$ of $G-G^i$, and $V_e=N(C)$. 
    As $G^i$ has finite adhesion in $G$ by \ref{theorem:displayPrescribed:item:finiteAdhesion}, the adhesion set $V_e$ is finite. Thus, $(T, \cV)$ has finite adhesion.
    
    Moreover, every bag corresponding to a node above $t_C$ is a subset of $V(C) \cup N(C)$, and, conversely, no other bag meets $V(C)$. Hence $G \strictup e$ is exactly $C$ by \cref{lem:displayPrescribed:limitIsG}, which implies that $(T, \cV)$ is tight and componental.
\end{proof}

It remains to show that $(T, \cV)$ displays $\Psi$, and that $(T, \cV)$ and $(\cD_e)_{e \in E(T)}$ satisfy \ref{item:EpsLinkedRegions} to \ref{Lproperty:pathLikeRegions}. We first verify the latter, as \ref{item:EpsLinkedRegions} and \ref{Lproperty:pathLikeRegions} will be useful to prove that $(T, \cV)$ displays $\Psi$.

Let us remark that given a region $D \in \cD_e$ for some edge $e=st \in T$, it holds that $D \subseteq G - G^i$, where $i=\max\{\tht(s), \tht(t)\}$. More precisely, $D \subseteq G \strictup e \subseteq G-G^i$. Hence, \cref{item:EpsLinkedRegions} and \cref{item:LinkedNhoods} follow from \cref{claim:regionsAlgo} together with \cref{observation:algoLinkedEnds}~\ref{property:everyEndInSomeC_i} and \ref{property:linkedToX}. For \ref{Lproperty:pathLikeRegions}, we have to show that

    \begin{claim}
        For every region $D \in \cD_e$ there exists a (unique) edge $f \in T$ with $e <_T f$, such that $D= G \strictup f$.
    \end{claim}
    \begin{proof}\renewcommand{\qedsymbol}{$\blacksquare$}
        Let $st=e \in E(T)$ be an arbitrary edge. It suffices to check that
        \begin{enumerate}[label=($*$)]
            \item \label{lem:displayLinked:sufficientDisplayRegion}for every region $D \in \bigcup_{g>_T e} \cD_e \setminus \cD_g$, the set $N(D)$ is contained in a unique bag $V_g$, for an edge $g>_T f$ with $g \cap f \neq \emptyset$.
        \end{enumerate}
            Indeed, condition 
            \cref{lem:displayLinked:sufficientDisplayRegion} is sufficient: Every region $D \in \cD_{st}$ is disjoint to $G^i$, where $i := \max \{ \tht(s), \tht(t) \}$. Given any $e \in E(T)$ and $D \in \cD_e$, one can fix an arbitrary vertex $x \in D$. Now, whenever $D$ is contained in a set $\cD_{g}$ it follows that $D \subseteq G \strictup e$ and hence $x \in G \strictup e$. If $D$ does not appear as $G \strictup f$ for any $f>_T e$, by inductively leveraging \cref{lem:displayLinked:sufficientDisplayRegion}, one finds a rooted ray $R$ of $T$ satisfying $x \in G \strictup f$ for every edge $f \in E(R)$, a contradiction.

            But \cref{lem:displayLinked:sufficientDisplayRegion} is satisfied by construction: Assume $s<_T t$ and let $D \in \cD_e$ be a region that is no element of any set $\cD_g$, for any edge $g>_T e$. Since $V_t$ is disjoint to any region in $\cD_e$, the region $D$ is contained in a unique component $G \strictup tx$, for some node $x \in T$ with $x>_T t$. By assumption, $D \notin \cD'_{tx}$ and hence $D = G \strictup tx$.
    \end{proof}

Finally, we conclude the proof by showing that

\begin{claim}
        $(T, \cV)$ displays $\Psi$.
    \end{claim}
    \begin{proof}\renewcommand{\qedsymbol}{$\blacksquare$}
        By \cref{claim:TightAndComponental} and \cref{lem:compDisplaysBoundary}, it suffices to check that the boundary of $(T, \cV)$ is $\Psi$. For this, we have to show that the boundary of any bag of $(T, \cV)$ contains only ends in $\Xi$, and every end contained in the boundary of $(T, \cV)$ is an element of $\Psi$. 
        \smallskip
        
        Let us first check that the boundary of every bag $V_t$ is a subset of~$\Xi$. For $t = \rt(T)$ this is obvious as $V_{\rt(T)}$ is a singleton, and hence its boundary is empty.
        For any other node $t \in T$, its bag $V_t$ is an envelope of $U_{t,1} \cup U_{t,2} \cup U_{t,3}$, and hence $\partial(V_t) = \partial(U_{t,1} \cup U_{t,2} \cup U_{t,3})$.
        As the sets $X_i$ are closed and $U_{t,1}, U_{t,2} \subseteq X_i$, the closures of $U_{t,1}$ and $U_{t,2}$ contain only ends in $\Xi$. Hence, it remains to consider $U_{t,3}$.
        
        For this, let $\eps \in \partial(U_{t,3})$. Further, let $f$ be the unique edge below $t$, and let $m$ the step in which we added $t$ to $T^{m}$. 
        To show that $\eps \in \Xi$, it suffices to prove that there is no finite $\eps$--$X_m$ separator~$S$ in $G$ (since $\partial(X_m) \subseteq \Xi$.).
        To this end, let $S \subseteq V(G)$ be any finite set, and let $H$ be a comb attached to $U_{t,3}$ with spine in $\eps$. 

        We need to show that $X_m$ meets the closure of the component $C(S, \eps)$ of $G-S$ in which $\eps$ lives. Since $S$ is finite, some subcomb $H'$ of $H$ is contained in $C(S, \eps)$. By definition of $U_{t,3}$ and because $H$ has all its teeth in $U_{t,3}$, every tooth of $H'$ is contained in the neighbourhood $N(D)$ of some region $D \in \cC_{f, \max}$ whose closure meets $X_m$. 
        Again since $S$ is finite and because the regions in $\cC_{f,\max}$ are pairwise disjoint, there is a tooth $x$ of $H'$ and a region $D \in \cC_{f,\max}$ such that $x \in N(D)$ and $D$ is disjoint to $S$ but its closure meets $X_m$.
        Since $x \in V(C(S, \eps)) \cap N(D)$ and because $D$ avoids $S$, we have that $D \subseteq C(S, \eps)$.
        As $\cl(D)$ meets $X_m$ and $D \subseteq C(S,\eps)$, it follows that $\cl(C(S,\eps))$ meets $X_m$, and hence $S$ does not separate $\eps$ from $X_m$. 
        This concludes the proof that the boundary of any bag of $(T, \cV)$ contains only ends in $\Xi$.
        \smallskip

        It remains to check that no end in $\Xi$ is contained in the boundary of $(T, \cV)$.
        For this, let $\xi \in \Xi$ be some end, and assume for a contradiction that $\xi$ lies in the boundary of $(T, \cV)$. Let $R_{\xi}$ be the rooted ray of $T$ corresponding to $\xi$. As $(T, \cV)$ and $(\cD_e)_{e \in E(T)}$ satisfy \cref{item:EpsLinkedRegions} and \cref{Lproperty:pathLikeRegions}, there are cofinally many edges $e \in R_{\xi}$ such that $V_e$ is linked to $\xi$. 
        Let $n \in \mathbb{N}$ be large enough such that $\xi \in X_n$, and let $e \in R_{\xi}$ be an edge such that $V_e$ is linked to $\xi$ and $e \notin T^n$. 
        It suffices to show that $\xi$ lives in no region in $\cD_e$, as then $\xi \in U_{t,2} \subseteq \partial(V_t)$ lives in the boundary of $V_t$, which yields the desired contradiction. 
        For this, recall that $\cD_e = \cD'_e \cup \cC_e$. As $\cC_f$ is an output of \cref{algo:RegionAlgo} on input $X = V_e$, all regions in $\cC_f$ are $(<|V_e|)$-regions. Since $V_e$ is linked to $\xi$, this implies that $\xi$ lives in no region in $\cC_f$. Similarly, also every region in $\cD'_e$ is an $(<|V_e|)$-region, as $\cD_f$, where $f$ is the unique edge in $T$ below $e$, is a potential output of \cref{algo:RegionAlgo}, and hence every region $D$ in $\cD_f$ with $D \subsetneq G\strictup e$ has a neighbourhood that is smaller than $V_e = N(G\strictup e)$. Hence, $\xi$ lives neither in a region in $\cD'_e$.
\end{proof}
    
    All in all, $(T, \cV)$ and $(\cD_e)_{e \in E(T)}$ are as desired. 
\end{proof}

We are now ready to prove \cref{thm:LinkedTD:Psi:Copy}.

\begin{proof}[Proof of \cref{thm:LinkedTD:Psi:Copy}]
    Let $(T, \cV)$ be the rooted \td\ provided by \cref{theorem:displayPrescribed} for~$G$ and $\Psi$.
    In particular, $(T, \cV)$ is tight and componental, and it displays $\Psi$.
    Let $L$ be the set of all edges $f \in T$ whose adhesion set $V_f$ is linked to some end living in $G\up f$. 
    Let $(T', \cV')$ be obtained from $(T, \cV)$ by contracting all edges of $T$ that are not in $L$, and letting $V'_{[t]} := \bigcup_{s \in [t]} V_s$ for every $t \in V(T)$.
    Clearly, $(T', \cV')$ is still a rooted \td, and it inherits the properties tight, componental and finite adhesion from $(T, \cV)$.
    We claim that $(T', \cV')$ is as desired, for which it remains to show that $(T', \cV')$ is linked and displays $\Psi$. 
    \smallskip
    
    We first check the latter. Since $(T', \cV')$ is tight and componental, it displays its boundary by \cref{lem:compDisplaysBoundary}. As we obtained $(T', \cV')$ from $(T, \cV)$ by contracting edges, its boundary is a subset of the boundary of $(T, \cV)$. Hence, as $(T, \cV)$ displays $\Psi$, it suffices to show that every end $\eps$ in the boundary of $(T, \cV)$ is also contained in the boundary of $(T', \cV')$. 
    This follows immediately once we have shown that cofinally many edges $e$ of $R_\eps$ (the (unique) rooted ray of $T$ whose end corresponds to $\eps$) are associated with an adhesion set $V_e$ that is linked to $\eps$ (as we did not contract those edges from $T$ to obtain $T'$). 
    
    So suppose for a contradiction that there is an edge $e$ of $R_\eps$ such that no adhesion set of an edge $e'$ above $e$ is linked to $\eps$, and let $e'$ be the next edge on $R_\eps$ after $e$. Then by \cref{item:EpsLinkedRegions}, some region $D \in \cD_{e'}$ is $\eps$-linked, and by \cref{Lproperty:pathLikeRegions}, $D = G\strictup f$ for some edge $f$ of $T$. In particular, $V_f$ is $\eps$-linked since $D$ is $\eps$-linked and $N(D) = V_f$ as $(T, \cV)$ is tight.
    Since $D \subseteq G\strictup e'$, it follows that $f$ lies above $e$ on~$R_\eps$, which contradicts the choice of~$e$.
    \smallskip

    It remains to check that $(T', \cV')$ is linked.
    Let $e \leq_{T} e' \in L$ be two comparable edges. We proceed by induction on the distance~$n$ between $e$ and $e'$ in $T$. For distance $n=0$ there is nothing to show. Now assume the distance $n$ is $>0$, and let $f \in L$ be the $\leq_{T}$-minimal edge with $e \leq_{T} f \leq_{T} e'$ and $|V_f| = \min \{ |V_g| \colon g \in E(eTe') \cap L \}$. If $f \neq e,e'$, then by the induction hypothesis, we find $|V_f|$ pairwise disjoint $V_e$--$V_f$ paths in $G \up e$ and $|V_f|$ pairwise disjoint $V_f$--$V_{e'}$ paths in $G \up f$. Combining the two path families yields our desired family of $V_e$--$V_{e'}$ paths.
     
    In the other case where $f \in \{e,e'\}$, assume for a contradiction that there is no family of $|V_f|$ pairwise disjoint $V_e$--$V_{e'}$ paths in $G \up e$. By Menger's Theorem, there is a $V_e$--$V_{e'}$ separator $S$ in $G \up e$ of size $<\min \{|V_e|, |V_e'| \}$. Since $e'\in L$, its adhesion set $V_{e'}$ is linked to an end $\eps$ living in $G \up e'$. Moreover, $V_e$ cannot be linked to $\eps$, as any $V_e$--$\eps$ path or ray meets $V_{e'}$ and hence $S$. So by \cref{item:EpsLinkedRegions}, some region $D \in \cD_e$ is $\eps$-linked, and by \cref{Lproperty:pathLikeRegions}, $D$ appears as $G \strictup g$ for an edge $g \in E(T')$. Since $D \subseteq G\strictup e$ as $D \in \cD_e$, it follows that $e \leq_{T} g$.
    Moreover, as $\eps$ lives in $G\up e'$, we have $g \leq_T e'$. In fact, $g <_T e'$ as the neighbourhood $N(D)$ of $D$ is linked to $V_e$ by \cref{item:LinkedNhoods}, and hence, witnessed by $S$, we have $g \neq e'$. 
    In particular, this implies that $|V_g| \leq |V_{e'}|$ as $V_g$ is linked to~$\eps$. But also $|V_g|=|N(D)| < |V_e|$ as $D \in \cD_e$, and hence $g$ was a better choice when we picked $f$, a contradiction.
\end{proof}

\begin{proof}[Proof of \cref{main:LinkedTD:Psi}]
    \ref{itm:main:Psi:i} $\Rightarrow$ \ref{itm:main:Psi:ii} is trivial. \ref{itm:main:Psi:ii} $\Rightarrow$ \ref{itemGdelta} is easy (see \cref{sec:Intro}). \ref{itemGdelta} $\Rightarrow$ \ref{itm:main:Psi:i} is \cref{thm:LinkedTD:Psi:Copy}.
\end{proof}

\begin{proof}[Proof of \cref{maincor:DisplayingEndDegrees1}]
    Apply \cref{main:LinkedTD:Psi} and \cref{lem:compDisplaysBoundary,lem:LinkedTightCompDisplaysDomAndDegrees}.
\end{proof}

\begin{proof}[Proof of \cref{maincor:DisplayingEndDegrees2}]
    The set of all undominated ends of $G$ is a $G_\delta$-set (see \cref{sec:Intro} or \cite{koloschin2023end}*{Theorem~10.1}). Apply \cref{maincor:DisplayingEndDegrees2} with $\Psi$ the set of all undominated ends of $G$, to obtain a \td\ $(T, \cV)$. Then in particular, $(T, \cV)$ displays the undominated ends of $G$, and their dominating vertices. This implies that for every node $t \in T$ and every ray $R \subseteq T$ starting at $t$, there is $t' \in R$ with $V_t \cap V_{t'} = \emptyset$, as desired.
\end{proof}

\bibliographystyle{amsplain}
\bibliography{collective.bib}

\arXivOrNot{
\appendix

\section{A counterexample regarding connected parts}
\label{Appendix:counterexampleConnectedParts}

   In this appendix we present a counterexample that shows that \ref{itm:main:Psi:i} of \cref{main:LinkedTD:Psi} cannot be strengthened so that the \td\ has connected parts. More precisely, we prove the following:

 \begin{example} \label{ex:ConnectedParts}
     \emph{There exists a locally finite graph which admits no linked, tight, componental, rooted \td~of finite adhesion with connected parts that displays all its (undominated) ends.}
 \end{example}

     For this, we need the following lemma, which essentially says that the linked property of a linked, tight, componental \td\ of finite adhesion extends to an edge-end pairs.

  \begin{lemma}\label{lem:linkedTdSeparator}
     Let $G$ be a graph, $(T,\cV)$ a linked, tight and componental rooted \td\ of $G$ of finite adhesion, and let $\eps$ be an end in the boundary of $(T, \cV)$. For any edge $e \in R_\eps$ there exists an edge $e<_Tf \in R_{\eps}$ such that $V_f$ is a $\size$-minimal $V_e$--$\eps$ separator.
 \end{lemma}
 
 \begin{proof}
     Let $e \in R_\eps$ be some edge, and let $X$ be a $\size$-minimal $V_e$--$\eps$ separator. As $X$ is finite, there is an edge $g \in R_\eps$, with $e<_Tg$, such that $C(V_g, \eps)$ is disjoint to $X$. Since $(T, \cV)$ is linked, there exists an edge $f \in eTg$ such that there are $|V_f|$ pairwise disjoint $V_e$--$V_g$ paths in $G$. Moreover, as $(T, \cV)$ is tight and componental, it follows that $V_g = N(G \strictup g)$ and $G \strictup g = C(V_g, \eps)$. Hence, every $V_e$--$V_g$ path can be extended in $G\strictup g$ to an $\eps$-ray, which implies that it must meet $X$, and thus $|V_f| \leq |X|$. On the other side, any $V_e$--$V_g$ separator is also a $V_e$--$\eps$ separator, and thus, witnessed by $X$, it follows  that $|X|\leq |V_f|$.
 \end{proof}

 We are now ready to prove \cref{ex:ConnectedParts}.

 \begin{proof}[Proof of \cref{ex:ConnectedParts}]
     Given $n \in \mathbb{N}$, let $H_n$ be the grid on $[n] \times \N$, that is the graph obtained by $n$ disjoint rays where we add edges between any two vertices on `neighbouring' rays that have the same distance to the starting vertex on its ray. More formally, $H_n$ is the graph with vertex set $[n] \times \mathbb{N}$ and edges between vertices $(i,j)$ and $(l,k)$ whenever $i=l$ and $|j-k|=1$, or $|i-j|=1$ and $j=k$.
     We call $X \subseteq V(H_n)$ a \defn{rung} of $H_n$ if $X$ is of the form $\{(i,j) : i \in [n]\}$ for some $j \in \N$ (see the left graph in \cref{fig:connectedParts}).

     Let $Q$ denote the graph depicted on the right in \cref{fig:connectedParts}. Formally, $Q=(V,E)$ is the graph with vertex set $V = Q_X\; \dot\cup\; Y\; \dot\cup\; (\{1\} \times [3] \times \mathbb{N})\; \dot\cup\; (\{2\} \times [4] \times \mathbb{N})$, where $Q_X := \{x_1, x_2, x_3, s_1\}$ and $Y := \{y_1, y_2, s_2 \}$ are `fresh' vertices in the sense that they are disjoint to any of the other vertices denoted in the union of $V$. Let $S := \{ s_1, s_2 \}$. The edges of $Q$ consists of all edges such that $Q[1 \times [3] \times \mathbb{N}]$ is isomorphic to an $H_3$, $Q[2 \times [4] \times \mathbb{N}]$ is isomorphic to an $H_4$, both with respect to the natural embedding, and the sets $E[Y, \{2\} \times [4] \times \{1\}]$, $E[S, \{1\} \times [3] \times \{1\}]$ and $E[\{x_1, x_2, x_3\},\{ y_1, y_2 \}]$ contain all possible edges, and $E$ includes the edges $x_1x_2, x_2x_3$ and $x_3s_1$. 

     \begin{figure}[htbp]
     \centering
     \includegraphics[scale=1]{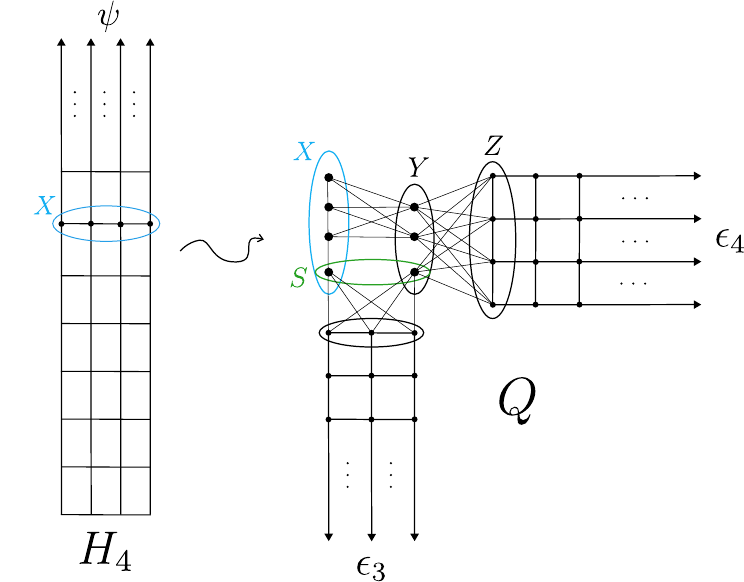}
     \caption{Depicted is a part of the construction of $G$.}
     \label{fig:connectedParts}
     \end{figure}
    
     For the construction of our counterexample $G$ we start with a copy $H$ of $H_4$. Let $G$ be the graph obtained by gluing, along every rung $X$ of $H_4$, a copy of $Q$ by identifying $X$ with $Q_X$ (see also \cref{fig:connectedParts} for a visualisation). Specifically, we identify $x_i \in Q_X$ with the unique vertex in $X$ with its first coordinate equal to $i$ for $i=1, \dots ,4$, where we set $x_4 := s_1$. Moreover, we glue this copy of $Q$ in such a way that it is disjoint to any other copy of $Q$ glued along a different rung of $H_4$ and this copy of $Q$ intersects $H_4$ in exactly $X_Q$. This completes the construction of $G$.
     \smallskip
    
     We claim that $G$ admits no linked, tight, componental, rooted \td~of finite adhesion with connected parts that displays all its ends. (Note that $G$ is locally finite, and hence all its ends are undominated.) 
     Assume for a contradiction that $(T, \cV)$ is such a \td. Denote by $\psi$ the unique end induced by $H_4 \subseteq G$, and let $R_\psi$ be the (unique) rooted ray of $T$ that `corresponds' to $\psi$. 
     As $\psi$ has degree four, the adhesion set $V_e$ of some edge $e$ on $R_\psi$ must have size at least four. Since $V_e$ is finite and $\psi$ lives in $G\strictup e$, there is a rung $X$ of $H_4$ in $G\strictup e$ such that the isomorphic copy of $Q$ glued along $X$ is a subset of $G\strictup e$. Let us denote by $Q$ the specific copy of $Q$ we glued along $X$. Further, let $X=Q_X$, $Y$, $S$ and $Z:=N(Y \setminus \{ s_2 \}) \setminus X$ be the subsets of $V(Q)$ as defined above (see also \cref{fig:connectedParts}). In particular, $S=\{s_1, s_2 \}$ where $\{s_1 \}= X \cap S$ and $\{ s_2 \} = Y \cap S$. Lastly, let $\eps_3$ denote the (unique) end of $Q$ of degree three and $\eps_4$ the (unique) end of $Q$ of degree four. 
     \smallskip

     Observe that the set $S$ is a $\size$-minimal $Q_X$--$\eps_3$ separator in $Q$. Since we glued $H$ and $Q$ together in $X = Q_X$, and as $V_e$ has size four, it is easy to check that $S$ is also a $\size$-minimal $V_e$--$\eps_3$ separator in $G$; in fact, it is the only such separator.
     As $S$ has size two, it follows by \cref{lem:linkedTdSeparator} that there exists an edge $f \in R_{\eps_3}$ with $e<_Tf$ and $|V_f|=2$.
     Since $S$ is the only $V_e$--$\eps$ separator of size two, it follows that $V_f = S$. Let $t$ be the (unique) node of $T$ below $f$. By assumption, $G[V_t]$ is connected. As $S = V_f \subseteq V_t$ and $s_1, s_2$ are not joined by an edge, it follows that $V_t$ contains at least one neighbour of $s_1$ and one of $s_2$. Since $t <_T f$, its bag $V_t$ does not contain any vertices from $G\strictup f = C(S, \eps_3)$, and hence $V_t$ must contains a vertex of $Z$ and a vertex of $X \setminus \{s_1\}$ (cf.\ \cref{fig:connectedParts}).

     Similarly, $Y$ is a $\size$-minimal $V_e$--$\eps_4$ separator in $G$, and hence, by \cref{lem:linkedTdSeparator} and because $|Y| = 3$, there exists an edge $g \in R_{\eps_4}$ with $e <_Tg$ and $|V_g|=3$. 
     As there are four disjoint $V_e$--$X$ paths in $G$, the separator $V_g$ must meet $V(Q)$. In fact, it is easy to check that there are exactly two possibilities for $V_g$: either $V_g = \{y_1, y_2, s_1 \}$ or $V_g = \{y_1, y_2, s_2 \}$.

     In the former case, both $\eps_3$ and $\eps_4$ live in $G\strictup g$, and hence $g$ must lie on both $R_{\eps_3}$ and $R_{\eps_4}$. As $V_f = S$ meets $G \strictup g$, it holds that $g <_T f$, and hence $g <_T t$. 
     Moreover, $G \strictup g = C(V_g, \eps_3)$ is disjoint from $X$, and hence $V_g$ must contain all vertices in $V_t \cap X$. But $V_g \cap X = \{s_1\}$, while $V_t$ contains a vertex of $X \setminus \{s_1\}$, a contradiction.

     In the latter case, $V_g$ separates $\eps_3$ from $\eps_4$, i.e.\ $C(V_g, \eps_3) \neq C(V_g, \eps_4)$. As $(T, \cV)$ is componental, $\eps_3$ does not live in $G\strictup g$, and hence $g$ does not lie on the ray $R_{\eps_3}$. In particular, $t$ and $g$ are $<_T$-incomparable. But $Z \subseteq G \strictup g$ and $V_t$ contains a vertex of $Z$, which implies that $g <_T t$, a contradiction.
 \end{proof}
}{}

\end{document}